\documentclass[12pt]{amsart}

\title{Multiplicative functions in short intervals}

\author{Kaisa Matom\"aki}
\address{Department of Mathematics and Statistics, University of Turku,
20014 Turku, Finland}
\email{ksmato@utu.fi}

\author{Maksym Radziwi\l\l}
\address{Department of Mathematics, Rutgers University \\ Hill Center for the Mathematical Sciences \\
110 Frelinghuysen Rd., Piscataway, NJ 08854-8019 }
\email{maksym.radziwill@gmail.com}

\usepackage{amssymb, amsthm, enumerate}
\usepackage{amsmath}
\newtheorem{theorem}{Theorem}
\newtheorem{lemma}{Lemma}
\newtheorem{proposition}{Proposition}
\newtheorem{corollary}{Corollary}

\theoremstyle{remark}
\newtheorem*{remark}{Remark}
\usepackage{geometry}

\begin{document}
\begin{abstract}
We introduce a general result relating  ``short averages'' of a multiplicative function to ``long averages'' which are well understood. This result has several consequences. First, for the M\"obius function we show that there are cancellations in the sum of $\mu(n)$ in almost all intervals of the form $[x, x + \psi(x)]$ with $\psi(x) \rightarrow \infty$ arbitrarily slowly. This goes beyond what was 
previously known conditionally on the Density
Hypothesis or the stronger Riemann Hypothesis. Second, we settle the long-standing conjecture on 
the existence of $x^{\epsilon}$-smooth numbers in intervals of the form $[x, x + c(\varepsilon) \sqrt{x}]$, recovering unconditionally a conditional 
(on the Riemann Hypothesis) result of Soundararajan. Third, we show that the mean-value of 
$\lambda(n)\lambda(n+1)$, with $\lambda(n)$ Liouville's function, is non-trivially bounded in absolute value by $1 - \delta$ for some 
$\delta > 0$. This settles an old folklore conjecture and constitutes progress towards Chowla's conjecture. 
Fourth, we show that a (general) real-valued multiplicative function $f$ has a positive proportion of sign changes if and only
if $f$ is negative on at least one integer and non-zero on a positive proportion of the integers. This improves on many
previous works, and is new already in the case of the M\"obius function. 
We also obtain some additional results on smooth numbers in almost all intervals, and sign changes of 
multiplicative functions in all intervals of square-root length.
\end{abstract}

\dedicatory{Dedicated to Andrew Granville}

\maketitle


\section{Introduction}
Let $f: \mathbb{N} \rightarrow [-1,1]$ be a multiplicative function.
We introduce a general result relating many ``short averages'' of a multiplicative function over a bounded length
interval to ``long averages'' which are well understood using tools from multiplicative number theory.
\begin{theorem} \label{thm:main} 
Let $f: \mathbb{N} \rightarrow [-1,1]$ be a multiplicative function. 
There exist absolute constants $C, C' > 1$ such that for any $2 \leq h \leq X$ and $\delta > 0$, 
\[
\Bigg | \frac{1}{h} \sum_{x \leq n \leq x + h} f(n) - \frac{1}{X} \sum_{X \leq n \leq 2X} f(n) \Bigg | \leq \delta + C'\frac{\log \log h}{\log h}
\]
for all but at most 
$$
C X \Big ( \frac{(\log h)^{1/3}}{\delta^2 h^{\delta/25}} + \frac{1}{\delta^2 (\log X)^{1/50}} \Big )
$$
integers $x \in [X, 2X]$. One can take $C' = 20000$.
\end{theorem}

Note that Theorem \ref{thm:main} allows $h, \delta$ and $f$ to vary uniformly.
For example taking $\delta = (\log h)^{-1/200}$ gives a saving of $2(\log h)^{-1/200}$ with an
exceptional set of at most $C X (\log h)^{-1/100}$.  
Already for the M\"obius function $\mu(n)$ Theorem \ref{thm:main} goes beyond what was previously known conditionally;
The density hypothesis implies that there are cancellations in the sum of 
$\mu(n)$, but ``only'' in almost all intervals $x \leq n \leq x + h$ of length $h \geq x^{\varepsilon}$ whereas the Riemann hypothesis implies cancellations of $\mu(n)$ in almost all intervals but again ``only'' if
$h > (\log X)^{A}$ for some constant $A > 0$ (by unpublished work of Peng Gao). Unconditionally, using results towards the density hypothesis, it was previously known that there are cancellation of $\mu(n)$ in almost all intervals of length $x^{1/6 + \varepsilon}$
(a result due to Ramachandra \cite{Ramachandra}).

One naturally wonders if it is possible to establish Theorem \ref{thm:main} in all intervals of
length $h \asymp \sqrt{X}$. However, this is not possible in general, since it would require us to
control the contribution of the large primes factors which is completely arbitrary for general $f$. 
We prove however a bilinear version of Theorem \ref{thm:main} which holds in all intervals of length $\asymp \sqrt{X}$.
The bilinear structure allows us to eliminate the contribution of the large primes. 
\begin{theorem} \label{thm:bilinear}
Let $f : \mathbb{N} \rightarrow [-1,1]$ be a multiplicative function. Then, for any $10 \leq h \leq x$, 
$$
\frac{1}{h \sqrt{x} \log 2} \sum_{\substack{x \leq n_1 n_2  \leq x + h \sqrt{x} \\  \sqrt{x} \leq n_1 \leq 2 \sqrt{x}}} f(n_1) f(n_2) = \Big ( \frac{1}{\sqrt{x}} \sum_{\sqrt{x} \leq n \leq 2 \sqrt{x}} f(n) \Big )^2 + O \Big ( \frac{\log\log h}{\log h}
+ \frac{1}{(\log x)^{1/100}} \Big ). 
$$
\end{theorem}
An important feature of Theorem \ref{thm:bilinear} is that it holds uniformly in $h$ and $f$. 
Theorem \ref{thm:bilinear}
allows us to show the existence of many $X^{\varepsilon}$ smooth numbers in intervals of length $\asymp \sqrt{X}$. 
Alternatively we could have deduced this from Theorem \ref{thm:main} using ideas of Croot \cite{CrootSmooth} (building on earlier work of Friedlander and Granville \cite{FriedlanderGranville}).
\begin{corollary} \label{cor:smooths}
Let $\varepsilon > 0$ be given. There exists a positive constant $C(\varepsilon)$ such that
the number of $X^{\varepsilon}$-smooth numbers in $[X, X + C(\varepsilon) \sqrt{X}]$ is
at least $\sqrt{X} (\log X)^{-4}$ for all large enough $X$. 
\end{corollary}
This recovers unconditionally a conditional (on the Riemann Hypothesis) result of Soundararajan \cite{SoundararajanPre}
and comes close to settling the long-standing conjecture that every interval
$[x, x + \sqrt{x}]$, with $x$ large enough, contains $x^{\varepsilon}$-smooth numbers (see for example \cite[Challenge Problem 2000 in Section 4]{GranvilleSurvey}). The later conjecture
is motivated by attempts at rigorously estimating the running time of 
Lenstra's elliptic curve factoring algorithm \cite[Section 6]{Lenstra}. Our result also improves on earlier
work of Croot \cite{CrootSmooth}, Matom\"aki \cite{MatomakiSmooth, MatomakiSmooths2} and Balog \cite{BalogSmooth}. Finally for small fixed $\varepsilon$, a more difficult to state variant of Theorem \ref{thm:bilinear} (see section 2) 
shows that $C(\varepsilon) = \rho(1/\varepsilon)^{-13}$ is admissible, where $\rho(u)$ is the Dickman-de Brujin function. In fact with a little additional work the constant $C(\varepsilon)$ can be reduced further to $\rho(1/\varepsilon)^{-7}$ 
and the exponent $4$ in $\sqrt{x} (\log x)^{-4}$ could be refined to $\log 4$.

Another corollary of Theorem \ref{thm:main} is related to Chowla's conjecture,
\begin{equation} \label{eq:liouville}
\frac{1}{X} \sum_{n \leq X} \lambda(n) \lambda(n+1) = o(1) \ , \ \text{as } x \rightarrow \infty
\end{equation}
with $\lambda(n):= (-1)^{\Omega(n)}$ Liouville's function. Chowla's conjecture is believed to be 
at least as deep as the twin prime conjecture \cite{Hildebrand}. This motivates the old folklore 
conjecture according to which the sum (\ref{eq:liouville}) is, for all $X$ large enough, bounded 
in absolute value by $\leq 1 - \delta$ 
for some $\delta > 0$. For example, Hildebrand writes in \cite{HildebrandReview}``one would naturally expect the above sum to be $o(x)$ when $x \rightarrow \infty$, but even the much weaker relation
$$
\liminf_{x \rightarrow \infty} \frac{1}{x} \sum_{n \leq x} \lambda(n)\lambda(n+1) < 1
$$
is not known and seems to be beyond reach of the present methods''. 
Theorem \ref{thm:main} allows us to settle
this conjecture in a stronger form.
\begin{corollary} \label{cor:chowla}
For every integer $h \geq 1$, there exists $\delta(h) > 0$ such that
$$
\frac{1}{X} \Bigg | \sum_{n \leq X} \lambda(n) \lambda(n + h) \Bigg | \leq 1 - \delta(h)
$$
for all large enough $X > 1$. 
In fact the same results holds for any completely multiplicative
function $f: \mathbb{N} \rightarrow [-1, 1]$ such that 
$f(n) < 0$ for some $n > 0$. 
\end{corollary}
For $h = 1$ Corollary \ref{cor:chowla} also holds for any multiplicative $f: \mathbb{N} \rightarrow [-1,1]$ which is completely multiplicative at the prime $2$ (this rules out, for example, the $f$ such that $f(2^k) = -1$ and $f(p^k) = 1$ for all $p \geq 3, k \geq 1$). The ternary analogue
of Corollary \ref{cor:chowla} concerning cancellations in the sum of
$\lambda(n)\lambda(n+1)\lambda(n+2)$ is surprisingly much easier; it is
stated as an exercise in Elliott's book \cite[Chapter 33]{Elliott} (see also
\cite{ElliottPaper} and \cite{Cassaigne}). 

Corollary \ref{cor:chowla} is closely related to the problem of counting sign changes of $f(n)$. 
Using Hal\'asz's theorem one can show that if $\sum_{f(p) < 0} 1/p = \infty$ and $f(n) \neq 0$ for a positive proportion of the integers $n$ then the non-zero
values of $f(n)$ are 
half of the time positive and half of the time negative (see \cite[Lemma 2.4]{MatRadHecke} or \cite[Lemma 3.3]{ElsGun}).
Since we expect $f(n)$ and
$f(n+1)$ to behave independently this suggests that, for non-vanishing $f$ such that $\sum_{f(p) < 0} 1/p = \infty$, there should be about $x/2$ sign changes among integers $n \leq x$. When $f$ is allowed to be zero we say that $f$ has $k$ sign changes in $[1,x]$ 
if there are integers $1 \leq n_1 < n_2 < \ldots < n_{k+1} \leq x$ such that
$f(n_i) \neq 0$ for all $i$ and $f(n_i), f(n_{i+1})$ are of opposite signs for all $i \leq k$. 
For \textit{non-lacunary} multiplicative $f$, i.e multiplicative $f$ such that $f(n) \neq 0$ 
on a positive proportion of the integers, we still expect $\asymp x$ sign changes in $[1,x]$. 
\begin{corollary}\label{cor:signchanges}
Let $f: \mathbb{N} \rightarrow \mathbb{R}$ be a multiplicative function. Then $f(n)$ has a positive 
proportion of sign changes 
if and only if $f(n) < 0$ for some integer $n > 0$ and $f(n) \neq 0$
for a positive proportion of integers $n$. 
\end{corollary}
There is a large literature on sign changes of multiplicative functions. 
For specific multiplicative functions Corollary \ref{cor:signchanges} improves on earlier
results for:
\begin{itemize}
\item The M\"obius function. The previous best result was due to Harman, Pintz and Wolke \cite{HPW85} 
who obtained more than $x / (\log x)^{7 + \varepsilon}$ sign changes for $n \leq x$, using Jutila's bounds towards the density hypothesis (\cite{Jutila}).
\item Coefficients of $L$-functions of high symmetric powers of holomorphic Hecke cusp forms. In
this setting the best previous result was $x^{\delta}$ sign changes with some $\delta < 1$ \cite{LauLiuWu10}. 
\item Fourier coefficients of holomorphic Hecke cusp forms. In this case Corollary \ref{cor:signchanges}
recovers a recent result of the authors \cite{MatRadHecke}.
\end{itemize}
As observed by Ghosh and Sarnak in \cite{GhoshSarnak},
the number of sign changes of $\lambda_f(n)$ for $n \leq k^{1/2}$ (with $k$ the weight of $f$) is
related to the number of zeros of $f$ on the vertical geodesic high in the cusp. 
A suitable variation of Corollary \ref{cor:signchanges} (again deduced from Theorem \ref{thm:main}) has 
consequences for this problem. These results are discussed in a paper by the authors and Steve Lester 
(see \cite{LesterMatomakiRadziwill}).

For general multiplicative functions, Corollary \ref{cor:signchanges} improves on earlier work of
Hildebrand \cite{Hildebrand} and Croot \cite{Croot}. Croot obtained $x \exp(- (\log x)^{1/2 + o(1)})$
sign changes for completely multiplicative non-vanishing functions. Hildebrand showed that there
exists an infinite (but quickly growing) subsequence $x_k$ such that  $f$ has more than $x_k (\log\log x_k)^{-4}$
sign changes on the integers $n \leq x_k$. 

Corollary \ref{cor:signchanges} suggests
that unless $f$ is non-negative, there should be few long clusters of consecutive integers at which $f$
is of the same sign. Our next corollary 
confirms this expectation. 
\begin{corollary} 
\label{cor:signchangsinints}
Let $f: \mathbb{N} \rightarrow \mathbb{R}$ be a multiplicative function. If $f(n) < 0$ for some integer $n$
and $f(n) \neq 0$ for a positive proportion of integers $n$, then, for any $\psi(x) \to \infty$, almost every interval 
$[x, x + \psi(x)]$ contains a sign change of $f$. 
\end{corollary}
This is an optimal result, since on probabilistic grounds we expect that for any fixed $h > 0$ 
there is a positive proportion of intervals $[x, x + h]$ of length $h$ on which $f$ is of the same
sign. 
We also have the following analogue of Corollary \ref{cor:signchangsinints} for all intervals of length $\asymp \sqrt{x}$. 
\begin{corollary} \label{cor:signchangesall} Let $f: \mathbb{N} \rightarrow \mathbb{R}$ be a completely multiplicative function. If
$f(n) < 0$ for some integer $n > 0$ and $f(n) \neq 0$ for a positive proportion of integers $n$, 
then there exists a constant $C > 0$ such that $f$ has a sign change in the interval $[x, x + C \sqrt{x}]$
for all large enough $x$.
\end{corollary}
As a consequence of Corollary \ref{cor:signchangesall} there exists a constant $C > 0$, 
such that every interval $[n, n + C \sqrt{n}]$ has a number with an even number of prime
factors, and one with an odd number of prime factors. 

Our methods may also be used to demonstrate the existence of smooth numbers in almost all
short intervals. 
It is well-known that the number of $X^{1/u}$ smooth numbers up to $X$ is asymptotically
$\rho(u) X$ with $\rho(u)$ denoting the Dickman-De Brujin function \cite{Tenenbaum}. 
We show that this remains true in almost all short intervals, with the interval as short
as possible. 
\begin{corollary} \label{cor:smoothinshorts}
Let $\psi(x) \rightarrow \infty$ and let $u > 0$ be given. Then, for almost all $x$
the number of $x^{1/u}$-smooth integers in $[x, x + \psi(x)]$ is asymptotically
$\rho(u) \psi(x)$. 
\end{corollary}
This improves on earlier work of Matom\"aki \cite{MatomakiSmooths2} and unpublished work of Hafner
\cite{Hafner}. It would be interesting, in view of applications towards the complexity of Lenstra's
elliptic curve factoring algorithm, to extend Corollary \ref{cor:smoothinshorts}
to significantly smoother numbers (and one would naturally need somewhat longer intervals $[x, x + \psi(x)]$ with
a $\psi(x)$ depending on the smoothness under consideration), 
even under the assumption of the Riemann Hypothesis. 


We end this introduction by discussing extensions and limitations of our main result. Theorem \ref{thm:main} and its variants do not hold for complex valued multiplicative functions as the example $f(p) = p^{it}$ shows. However, the result does extend to complex-valued functions which are not $n^{it}$-pretentious. We carried out this extension in \cite{tao} (joint with Terence Tao), where we used this complex variant, together with other ideas, to prove
an averaged version of Chowla's conjecture. 

It is also interesting to notice that one cannot hope to establish general results on sign changes of a multiplicative function 
$f: \mathbb{N} \rightarrow \mathbb{R}$ in all short intervals $[x, x + y(x)]$ with $y(x) < \exp(((2 + o(1)) \log x \log\log x)^{1/2})$. 
Indeed in an interval of this length every integer might be divisible by a distinct prime factor. Therefore one can rig the sign of the
multiplicative function on those primes so that $f(n)$ is always positive in $[x, x + y(x)]$ even though $f(n)$ has many sign changes in
the full interval $[x, 2x]$. 

In forthcoming work, the authors will investigate versions of our results for multiplicative functions vanishing on a positive proportion of the primes. This is naturally related to sieves of small dimensions. In addition we will also look at the related question of what happens when $|f(p)|$ is not bounded by $1$. In particular we will obtain results for the $k$-fold divisor function. 
In another forthcoming work, related to Theorem \ref{thm:bilinear} and joint with Andrew Granville and Adam Harper, we will try to understand individual averages of a multiplicative function $f$ in intervals of length $x^{\theta}$ with $\theta > 1/2$, and with $n$ restricted to smooth numbers (thus eliminating the contribution of large primes). 

\section{Initial reduction and key ideas}
\label{se:KeyIdeas}
We will deduce Theorem~\ref{thm:main} from a variant where $n$ is restricted to a dense subset $\mathcal{S}_X \subset [X, 2X]$ which contains only those $n$ which have prime divisors from certain convenient ranges. To define the set $\mathcal{S}$ we need to introduce some notation.
Let $\eta \in (0, 1/6)$. Consider a sequence of increasing intervals $[P_j, Q_j]$ such that 
\begin{itemize}
\item $Q_1 \leq \exp(\sqrt{\log X})$.
\item The intervals are not too far from each other, precisely
\begin{equation}
\label{eq:PjQjnottoofar}
\frac{\log\log Q_{j}}{\log P_{j-1}-1} \leq \frac{\eta}{4j^2}.
\end{equation}
\item The intervals are not too close to each other, precisely
\begin{equation}
\label{eq:PjQjnottooclose}
\frac{\eta}{j^2} \log P_{j} \geq 8 \log Q_{j-1} + 16 \log j 
\end{equation}
\end{itemize}
For example, given $0 < \eta < 1/6$ choose any $[P_1, Q_1]$ with $\exp(\sqrt{\log X}) \geq Q_1 \geq P_1 \geq (\log Q_1)^{40/\eta}$ large enough, and choose the remaining $[P_j, Q_j]$ as follows:
\begin{equation}
\label{eq:PjQjchoice}
P_j = \exp(j^{4j} (\log Q_1)^{j-1} \log P_1) \quad \text{and} \quad
Q_j = \exp(j^{4j + 2} (\log Q_1)^j).
\end{equation}
Let $\mathcal{S} = \mathcal{S}_X$ be a set of integers $X \leq n \leq 2X$
having at least one prime factor
in each of the intervals $[P_j, Q_j]$ for $j \leq J$, where $J$
is chosen to be the largest index $j$ such that $Q_j \leq \exp((\log X)^{1/2})$. 


Notice that, for any $j \leq J$, the number of integers in $[X, 2X]$ that do not have a prime factor from $[P_j, Q_j]$ is by a standard sieve bound of order $X \frac{\log P_j}{\log Q_j}$, which with the choice~\eqref{eq:PjQjchoice} is $X\frac{\log P_1}{j^2 \log Q_1}$. Hence once $Q_1$ is large enough in terms of $P_1$, most integers in $[X, 2X]$ belong to $\mathcal{S}$. It is also worth noticing that with the choice (\ref{eq:PjQjchoice}) a typical integer has about $\log \frac{\log Q_j}{\log P_j} = 2 \log j + \log \log Q_1 - \log\log P_1$ distinct prime factors in every fixed interval $[P_j, Q_j]$. 

We will establish the following variant of Theorem~\ref{thm:main} on the integers $n \in \mathcal{S}$.
\begin{theorem}
\label{th:ThminS}
Let $f : \mathbb{N} \rightarrow [-1,1]$ be a multiplicative function. Let $\mathcal{S} = \mathcal{S}_X$ be as above with $\eta \in (0, 1/6)$. If $[P_1, Q_1] \subset [1, h]$, then for all $X > X(\eta)$ large enough
$$
\frac{1}{X}\int_X^{2X} \left|\frac{1}{h} \sum_{\substack{x \leq n \leq x + h \\ n \in \mathcal{S}}} f(n) - \frac{1}{X} \sum_{\substack{X \leq n \leq 2X \\
n \in \mathcal{S}}} f(n)\right|^2 dx \ll \frac{(\log h)^{1/3}}{P_1^{1/6-\eta}} + \frac{1}{(\log X)^{1/50}}.
$$
\end{theorem}
We show in Section \ref{sec:MTreduction} that for an appropriate choice of $\mathcal{S}$ almost all integers $n \in [X, 2X]$ belong to $\mathcal{S}$. It follows by taking $f(n) = 1$ in Theorem \ref{th:ThminS} that the same property holds in almost all short intervals. Combining this observation with Theorem \ref{th:ThminS}, and the assumption that $|f(n)| \leq 1$
implies Theorem \ref{thm:main}. 

To prove Theorem~\ref{thm:bilinear} we will establish the following variant on the integers $n_1, n_2 \in \mathcal{S}$.
\begin{theorem}
\label{th:longints}
Let $f: \mathbb{N} \rightarrow [-1,1]$ be a multiplicative function. Let $\mathcal{S}$ be as above with $\eta \in (0, 1/6)$. If $[P_1, Q_1] \subset [1,h]$, then for all $x > x(\eta)$ large enough
\begin{align*}
\frac{1}{h \sqrt{x} \log 2}  \sum_{\substack{x \leq n_1 n_2 \leq x + h \sqrt{x} \\ \sqrt{x} \leq n_1 \leq 2 \sqrt{x} \\ n_1, n_2 \in \mathcal{S}}}
f(n_1) f(n_2) = \Big (\frac{1}{\sqrt{x}} \sum_{\substack{\sqrt{x} \leq n \leq 2 \sqrt{x} \\ n \in \mathcal{S}}} f(n) \Big )^{2} 
 + O \Big ( \frac{(\log Q_1)^{1/6}}{P_1^{1/12 - \eta/2}} + (\log X)^{-1/100} \Big ). 
\end{align*}
\end{theorem}
As before, upon specializing the set $\mathcal{S}$ and sieving, we can get rid of the requirement that $n_1, n_2 \in \mathcal{S}$, thus obtaining Theorem \ref{thm:bilinear}. While Theorem \ref{th:longints} is more complicated
than Theorem \ref{thm:bilinear}, it outperforms the latter in certain applications, such as for
example estimating the constant $C(\varepsilon)$ in Corollary \ref{cor:smooths}. Using Theorem \ref{th:longints}
gives $C(\varepsilon) = \rho(1/\varepsilon)^{-13}$ in Corollary \ref{cor:smooths}, for small fixed $\varepsilon$, 
while Theorem \ref{th:longints}
would only give estimates of the form $C(\varepsilon) = \exp(c/\rho(1/\varepsilon))$. In addition, by using a smoothing
in Theorem \ref{th:longints}, one could further reduce the estimate for $C(\varepsilon)$ to $\rho(1/\varepsilon)^{-7}$ 
for small fixed $\varepsilon$. Similarly using Theorem \ref{th:ThminS} instead of Theorem \ref{thm:main} allows us to give a better bound in Corollary \ref{cor:signchangsinints} for the exceptional set
$\mathcal{E} \subset[X, 2X]$ of those $x$'s for which $[x, x + h]$ has no sign change of $f$. Indeed we can show using Theorem \ref{th:ThminS} that 
$\mathcal{E}$ has measure $O_{\varepsilon} (X h^{-1/6 + \varepsilon} + (\log X)^{-1/50})$. 

\subsection{Outline of the proofs of Theorems \ref{th:ThminS} and \ref{th:longints}}
We now discuss the ideas behind the proofs of Theorems~\ref{th:ThminS} and \ref{th:longints}. 
In both cases the first step consists in reducing the problem essentially to showing that
\begin{equation} \label{mainthingtobound}
\int_{(\log X)^{1/15}}^{X/h} \left| \sum_{\substack{X \leq n \leq 2 X \\ n \in \mathcal{S}}} \frac{f(n)}{n^{1+it}}\right|^2 dt \ll \frac{(\log h)^{1/3}}{P_1^{1/6-\eta}} + \frac{1}{(\log X)^{1/50}}.
\end{equation}
The above bound is established in Proposition \ref{prop:MainProp} in Section~\ref{sec:mainProp}, and we will now sketch how to prove this bound.
We caution the reader that in the actual proof of Proposition \ref{prop:MainProp} we need to argue more carefully and 
in particular split most Dirichlet polynomials into much shorter ranges to avoid an accumulation of error terms.

We begin by splitting the range of integration $(\log X)^{1/15} \leq t \leq X / h$ into $J+1$ disjoint sets $\mathcal{T}_1, \ldots, \mathcal{T}_J, \mathcal{U}$ which are defined according to 
the sizes of the Dirichlet polynomials 
\begin{equation} \label{eq:dirpoly101}
\sum_{P_j \leq p \leq Q_j} \frac{f(p)}{p^{1+it}}.
\end{equation}
More precisely, we will define $\mathcal{T}_j$ as follows: $t \in \mathcal{T}_j$ if $j$ is the smallest index such that all appropriate subdivisions of \eqref{eq:dirpoly101}, i.e 
$$
\sum_{P \leq p \leq Q} \frac{f(p)}{p^{1 + it}} \text{ with } [P,Q] \subset [P_j, Q_j]
$$ are small (i.e with an appropriate power-saving). In practice the ``sub-divisions'' $[P,Q]$ will be 
narrow intervals covering $[P_j, Q_j]$. We will also define $\mathcal{U}$ as follows: $t \in \mathcal{U}$ if there does not exists a $j$ such that $t \in \mathcal{T}_j$. The set $\mathcal{U}$ is rather sparse (its measure is $O(T^{1/2 - \varepsilon})$) and therefore $t \in \mathcal{U}$ can be considered an exceptional case. 
The argument then splits into two distinct parts. 

The first is concerned with obtaining a saving for
\begin{equation} \label{eq:intgoal}
\int_{\mathcal{T}_j} \Bigg | \sum_{\substack{X \leq n \leq 2X \\ n \in \mathcal{S}}} \frac{f(n)}{n^{1 + it}} \Bigg |^2 dt
\end{equation}
for each $1 \leq j \leq J$,
and the second part of the argument is concerned with bounding
\begin{equation} \label{eq:intgoal2}
\int_{\mathcal{U}} \Bigg | \sum_{\substack{X \leq n \leq 2X \\ n \in \mathcal{S}}} \frac{f(n)}{n^{1 + it}} \Bigg |^2 dt.
\end{equation}
The smaller the length of the interval $h$ is, the more sets $\mathcal{T}_j$ we are required to work with, which
leads to an increasing complication of the proof. It is worth mentioning that for intervals of length $h = X^{\varepsilon}$ it
is enough to take $J = 1$ and most of the work consists in dealing with $\mathcal{U}$. In addition, in the special
case $h = X^{\varepsilon}$ and $f(n) = \mu(n)$ we do not even need to consider the integral over $\mathcal{U}$
and a very simple argument suffices. Both of the above remarks are explained in detail in our short note \cite{ShortNote}. 

When $t \in \mathcal{T}_j$ we use an analogue of Buchstab's identity (a variant of Ramar\'e's identity \cite[Section 17.3]{Opera}) to extract from the Dirichlet polynomial
$$
\sum_{\substack{X \leq n \leq 2X \\ n \in \mathcal{S}}} \frac{f(n)}{n^{1+it}}
$$
a Dirichlet polynomial over the primes in $[P_j, Q_j]$, which is known to be small (by our assumption that $t \in \mathcal{T}_j$). 
More precisely, for completely multiplicative $f(n)$ (the same ideas works for general multiplicative functions, but is more transparent in this case) we have
\begin{align}
\label{eq:Buchstab}
\sum_{\substack{X \leq n \leq 2 X \\ n \in \mathcal{S}}} \frac{f(n)}{n^{1+it}} =
\sum_{P_j \leq p \leq Q_j} \frac{f(p)}{p^{1+it}} & \sum_{\substack{X/p \leq m \leq 2X/p \\ m \in \mathcal{S}_j}} \frac{f(m)}{m^{1+it}} \cdot \frac{1}{\# \{P_j \leq q \leq Q_j: q | m \} + \mathbf{1}_{(p,m) = 1}}, 
\end{align}
where $\mathbf{1}_{(p,m) = 1}$ is\footnote{In the published version of this paper the term $\mathbf{1}_{(p,m) = 1}$ was incorrectly expressed as $1$ leading to a slight gap in the argument that affected only the proof of Lemma 12. This is corrected here. We thank Alisa Sedunova and Ke Wang for pointing out this issue to us.} the indicator function of $(p,m) = 1$, and $\mathcal{S}_j$ is the set of integers which have a prime factor from each interval $[P_i, Q_i]$ with $i \leq J$ 
except possibly not from $[P_j, Q_j]$. We then do some cosmetic operations: we dispose of the condition $X / p \leq m \leq 2X / p$ by splitting into short segments, and we replace $\mathbf{1}_{(p,m) = 1}$ by $1$ by noticing that for most $t$ the contribution of the terms with $p | m$ is negligible. The next step is to use a pointwise bound (which follows from the definition of $\mathcal{T}_j$) for the polynomial over $p \in [P_j, Q_j]$ and a mean value theorem for Dirichlet polynomials for the remaining polynomial over $m$ (by forgetting about the condition $t \in \mathcal{T}_j$ and extending the range of integration to $|t| \leq X / h$). This gives the desired saving in \eqref{eq:intgoal} when $j = 1$, but for $j > 1$ the length of the Dirichlet polynomial
\begin{equation} \label{eq:dirpoly102}
R_P(1+it) = \sum_{\substack{X/P \leq m \leq 2 X / P \\ m \in \mathcal{S}_j}} \frac{f(m)}{m^{1+it}} \cdot \frac{1}{\#\{P_j \leq p \leq Q_j : p | m\} + 1} \ , \ P \in [P_j, Q_j]
\end{equation}
is too short compared to the length of integration to produce a good bound. To get around this issue, we will
use the definition of $\mathcal{T}_j$, namely the assumption that there exists a narrow interval $[P, Q] \subset [P_{j-1}, Q_{j-1}]$ for which
$$
\sum_{P \leq p \leq Q} \frac{f(p)}{p^{1+it}}
$$
is large, say $\geq V$. This allows us to bound the mean-value of \eqref{eq:dirpoly102} by the mean-value of
\begin{equation} \label{eq:thedirpoly}
\Bigg ( V^{-1} \sum_{P \leq p \leq Q} \frac{f(p)}{p^{1 + it}} \Bigg )^{\ell} R_P(1+it)
\end{equation}
with an appropriate choice of $\ell$, making the length of the above Dirichlet polynomial close to 
$X / h$ (which is also the length of integration). While computing the moments, the conditions~\eqref{eq:PjQjnottoofar} and~\eqref{eq:PjQjnottooclose} on
$[P_j, Q_j]$ arise naturally: $Q_{j-1}$
needs to be comparatively small with respect to $P_j$ so that the length of the Dirichlet polynomial
\eqref{eq:thedirpoly} is necessarily close to $X/h$ for some choice of $\ell$. On the other hand $Q_{j-1}$ cannot be too small compared to $P_j$, so that we are not forced to choose too large $\ell$
which would increase too much the mean-value of \eqref{eq:thedirpoly}. 
Fortunately, it turns out that neither condition is very restrictive and there is a large
set of choices of $[P_j, Q_j]$ meeting both conditions. 

Let us now explain how one bounds the remaining integral \eqref{eq:intgoal2}. In this case we split the Dirichlet
polynomial
$$
\sum_{\substack{X \leq n \leq 2X \\ n \in \mathcal{S}}} \frac{f(n)}{n^{1+it}}
$$
into a Dirichlet polynomial whose coefficients are supported on the integers which have a prime factor in the range $\exp((\log X)^{1-1/48}) \leq p
\leq \exp(\log X / \log\log X)$, say, and a Dirichlet polynomial whose coefficients are supported on the integers which are co-prime to every prime in this range. 
The coefficients of the second Dirichlet polynomial are supported on a set of smaller density, and applying the mean-value
theorem easily shows that we can ignore its contribution. To the first Dirichlet polynomial we apply the version of Buchstab's identity discussed before. In addition since $\mathcal{U}$ is a thin set (of size $O(T^{1/2 - \varepsilon})$) we can bound the integral by a sum of $O(T^{1/2 - \varepsilon})$ well-spaced points. Thus our problem reduces essentially to bounding
\begin{equation} \label{eq:Ssetbound}
(\log X)^{2 + \varepsilon} \sum_{t \in \mathcal{T}} |P(1 + it)M(1 + it)|^2 
\end{equation}
where $\mathcal{T}$ is a set of well-spaced points from $\mathcal{U}$, where $P(1 + it)$ is a polynomial whose coefficients are supported on the primes in a dyadic range, $M(1 + it)$ is the corresponding  Dirichlet polynomial
over the integers arising from Buchstab's identity, and the term $(\log X)^{2 + \varepsilon}$ comes from the loss incurred by ensuring that $P$ is in a dyadic interval. 

The Dirichlet
polynomial $|P(1 + it)|$ is small most of the time (in fact for $f = \mu$ it
is \textit{always} small for $|t| \leq X$), and on the set where it is small
we are done by simply bounding $P$ and 
applying Hal\'asz's large value estimate to sum $|M(1 + it)|^2$ over the well-spaced points $t \in \mathcal{T}$ (Hal\'asz's large values theorem is applicable since $|\mathcal{T}| \ll T^{1/2 - \varepsilon}$).
 On the
other hand taking moments we can show that $|P(1 + it)|$ is large extremely rarely (on a set of size $\exp((\log X)^{1/48+o(1)})$). We know in addition that $|M(1 + it)|^2$ is always $\ll (\log X)^{-\delta}$, for some small fixed $\delta > 0$, by Hal\'asz's theorem on multiplicative functions (since $f \in \mathbb{R}$ and
$|t| > (\log X)^{1/15}$ is bounded away from zero). Applying this pointwise bound to $|M(1 + it)|^2$ we are left with averaging $|P(1 + it)|^2$ over a very sparse set of points, and we need to save one logarithm compared to the standard application of Hal\'asz's large value estimate (which already regains one logarithm from the mean square of coefficients of $P$ since the coefficients are supported on primes in a dyadic interval). To do this, we derive a Hal\'asz type large value estimates for Dirichlet polynomials whose coefficients are supported on the primes. Altogether we regain the loss of $(\log x)^{2}$ and we win by $(\log x)^{-\delta + \varepsilon}$ which
followed from Hal\'asz's theorem on multiplicative functions. 

Finally, 
we note that an iterative decomposition of Dirichlet polynomials is employed in a different way in two very recent papers on moments of $L$-functions (see \cite{RadziwillSoundararajan} and \cite{Harper}). 

\section{Hal\'asz theorem}

As explained above, in the proof we use Hal\'asz's theorem which says that unless a multiplicative function pretends to be $p^{it}$, it is small on average. Pretending is measured through the distance function
\[
\mathbb{D}(f, g;x)^2 = \sum_{p \leq x} \frac{1-\Re f(p) \overline{g(p)}}{p}
\]
which satisfies the triangle inequality
\[
\mathbb{D}(f, h; x) \leq \mathbb{D}(f, g; x) + \mathbb{D}(g, h; x)
\]
for any $f,g,h : \mathbb{N} \rightarrow \{z \in \mathbb{C}: |z| \leq 1\}$.

Upon noticing that $\mathbb{D}(fp^{-it}, p^{it_0};x) = \mathbb{D}(f, p^{it+it_0};x)$, the following lemma follows immediately from Hal\'asz's theorem (see for instance~\cite[Corollary 1]{GS03}) and partial summation.
\begin{lemma}
\label{le:Hal}
Let $f \colon \mathbb{N} \to [-1, 1]$ be a multiplicative function, and let
\[
F(s) = \sum_{x \leq n \leq 2x} \frac{f(n)}{n^{s}}.
\] 
and $T_0 \geq 1$. Let
\[
M(x, T_0) = \min_{|t_0| \leq T_0} \mathbb{D}(f, p^{it+it_0}; x)^2
\]
Then
\[
|F(\sigma + it)| \ll x^{1-\sigma}\left(M(x, T_0)\exp(-M(x, T_0)) + \frac{1}{T_0} + \frac{\log \log x}{\log x} \right)
\]
\end{lemma}

The following lemma which is essentially due to Granville and Soundararajan is used to get a lower bound for the distance.
\begin{lemma}
\label{le:distest}
Let $f \colon \mathbb{N} \to [-1, 1]$ be a multiplicative function, and let $\varepsilon > 0$.
For any fixed $A$ and $1 \leq |\alpha| \leq x^A$,
\[
\mathbb{D}(f, p^{i \alpha}; x) \geq \left(\frac{1}{2\sqrt{3}}-\varepsilon\right) \sqrt{\log \log x} + O(1).
\]
\end{lemma}
\begin{proof}

By the triangle inequality
\[
2 \mathbb{D}(f, p^{i \alpha}; x) = \mathbb{D}(p^{-i \alpha}, f; x) + \mathbb{D}(f, p^{i \alpha}; x) \geq \mathbb{D}(p^{-i\alpha}, p^{i\alpha}; x) =  \mathbb{D}(1, p^{2i\alpha}; x).
\]
Furthermore
\[
\begin{split}
\mathbb{D}(1, p^{2i\alpha};x)^2 &= \sum_{p \leq x} \frac{1-\Re p^{-2i\alpha}}{p} \geq \sum_{\exp((\log x)^{2/3+\varepsilon}) \leq p \leq x} \frac{1-\Re p^{-2i\alpha}}{p} \\
&\geq \left(\frac{1}{3}-\varepsilon\right) \log \log x +O(1) - \left| \sum_{\exp((\log x)^{2/3+\varepsilon}) \leq p \leq x} \frac{1}{p^{1+2i\alpha}}\right| \\
&\geq \left(\frac{1}{3}-\varepsilon\right) \log \log x +O(1)
\end{split}
\]
by the zero-free region for the Riemann zeta-function.
\end{proof}

Actually we will need to apply Hal\'asz theorem to a function which is not quite multiplicative and the following lemma takes care of this application to a polynomial arising from the Buchstab type identity~\eqref{eq:Buchstab}.
\begin{lemma}
\label{le:Halappl}
Let $X \geq Q \geq P \geq 2$. Let $f(n)$ be a real-valued multiplicative function and
\[
R(s) = \sum_{\substack{X \leq n \leq 2X}} \frac{f(n)}{n^s} \cdot \frac{1}{\# \{p \in [P, Q] \colon p \mid n\} + 1}.
\]
Then for any $t \in [(\log X)^{1/16}, X^A]$,
\[
|R(1+it)| \ll  \frac{\log Q}{(\log X)^{1/16} \log P} + \log X \cdot \exp\left(-\frac{\log X}{3\log Q}\log \frac{\log X}{\log Q} \right).
\]
\end{lemma}
\begin{proof}
Splitting $n = n_1 n_2$ where $n_1$ has all prime factors from $[P,Q]$ and $n_2$ has none, we get
\[
\begin{split}
|R(1+it)| &= \left|\sum_{\substack{n_1 \leq X^{3/4} \\ p \mid n_1 \implies p \in [P, Q]}} \frac{f(n_1)}{n_1^{1+it} (\omega(n_1)+1)} \sum_{\substack{X/n_1 \leq n_2 \leq 2X/n_1 \\ p \mid n_2 \implies p \not \in [P, Q]}} \frac{f(n_2)}{n_2^{1+it}}\right| \\
\qquad &+ O\left(\sum_{\substack{n_2 \leq X^{1/2} \\ p \mid n_2 \implies p \not \in [P, Q]}} \frac{1}{n_2} \sum_{\substack{X/n_2 \leq n_1 \leq 2X/n_2 \\ p \mid n_1 \implies p \in [P, Q]}} \frac{1}{n_1}\right) \\
&\ll \sum_{\substack{n_1 \leq X^{3/4} \\ p \mid n_1 \implies p \in [P,Q]}} \frac{1}{n_1} \left|\sum_{\substack{X/n_1 \leq n_2 \leq 2X/n_1 \\ p \mid n_2 \implies p \not \in [P, Q]}} \frac{f(n_2)}{n_2^{1+it}}\right| + \sum_{n_2 \leq X^{1/2}} \frac{1}{n_2} \sum_{\substack{X/n_2 \leq n_1 \leq 2X/n_2 \\ p \mid n_1 \implies p <  Q}} \frac{1}{n_1}
\end{split}
\]
By an estimate for the number of $Q$-smooth numbers, the second term is at most $O((\log X)^{-1} + \log X \exp(-\frac{\log X}{3\log Q} \log \frac{\log X}{\log Q} ))$. To the first term we apply Hal\'asz's theorem (Lemmas~\ref{le:Hal} and~\ref{le:distest}) to the sum over $n_2$ obtaining a saving of $(\log X)^{-1/16}$ and we bound the sum over $n_1$ by $\prod_{p \in [P, Q]} (1-1/p)^{-1} \ll\frac{\log Q}{\log P}$. Hence
\[
|R(1+it)| \ll  \frac{\log Q}{(\log X)^{1/16} \log P} + (\log X) \exp\left(-\frac{\log X}{3\log Q} \log \frac{\log X}{\log Q}\right).
\]
\end{proof}

We will also evaluate the average of $f(n)$ on intervals slightly shorter than $x$. For this we use the following Lipschitz type result due to Granville and Soundararajan. 
\begin{lemma}
\label{le:Lipschitz}
Let $f \colon \mathbb{N} \to [-1,1]$ be a multiplicative function. For any $x \in [X, 2X]$ and $X/(\log X)^{1/5} \leq y \leq X$, one has
\[
\frac{1}{y}\sum_{x \leq n \leq x + y} f(n) = \frac{1}{X} \sum_{X \leq n \leq 2X} f(n) + O\left(\frac{1}{(\log X)^{1/20}}\right).
\]
\end{lemma}
\begin{proof}
We shall show that, for any $X/4 \leq Y \leq X$,
\begin{equation}
\label{eq:Lipsch}
\left|\frac{1}{X} \sum_{n \leq X} f(n) - \frac{1}{Y} \sum_{n \leq Y} f(n)\right| \ll \frac{1}{(\log X)^{1/4}}
\end{equation}
from which the claim follows easily.

Let $t_f$ be the $t$ for which $\mathbb{D}(f, p^{it}; X)$ is minimal among $|t| \leq \log X$. Notice that if $\mathbb{D}(f, p^{it_f}; X)^2 \geq \frac{1}{3} \log \log X$, then \eqref{eq:Lipsch} follows immediately from Hal\'asz's theorem (Lemma~\ref{le:Hal}). This is in particular the case if $|t_f| \geq 1/100$, since in this case
\[
\begin{split}
\mathbb{D}(f, p^{it_f}; X)^2 &\geq \sum_{p \leq X} \frac{1-|\cos(t_f \log p)|}{p} \geq \left(1-\frac{1}{2\pi}\int_{0}^{2\pi} |\cos \alpha| d\alpha -o(1) \right) \log \log X \\
&= \left(1-\frac{2}{\pi}-o(1)\right)\log \log X
\end{split}
\]
by partial summation and the prime number theorem. 

Hence we can assume that $|t_f| \leq 1/100$ and $\mathbb{D}(f, p^{it_f}; X)^2 < \frac{1}{3} \log \log X$. By~\cite[Lemma 7.1 and Theorem 4]{GS03}, recalling that $f$ is real-valued,
\begin{equation}
\label{eq:Lipcons}
\begin{split}
&\left|\frac{1}{X} \sum_{n \leq X} f(n) - \left(\frac{X}{Y}\right)^{it_f}\cdot \frac{1}{Y} \sum_{n \leq Y} f(n)\right| \\
&= \left|\frac{1}{X^{1+it_f}} \sum_{n \leq X} f(n) - \frac{1}{Y^{1+it_f}} \sum_{n \leq Y} f(n)\right| \\ 
& = \left|\frac{1+it_f}{X} \sum_{n \leq X} \frac{f(n)}{n^{it_f}} - \frac{1+it_f}{Y} \sum_{n \leq Y} \frac{f(n)}{n^{it_f}} \right| + O\left(\frac{1}{\log X} \exp(\mathbb{D}(1, f; X)^2)\right) \\ 
&\ll \frac{1}{(\log X)^{1/4}}.
\end{split}
\end{equation}
For $|t_f| \leq 1/100$ we have $|(X/Y)^{it_f} - 1| \leq 1/2$, so that~\eqref{eq:Lipcons} implies
\[
\left|\frac{1}{X} \sum_{n \leq X} f(n) - \frac{1}{Y} \sum_{n \leq Y} f(n)\right| \leq \frac{1}{2} \cdot \frac{1}{Y} \sum_{n \leq Y} f(n) + O((\log X)^{-1/4}),
\]
which implies that either the left hand side is $O((\log X)^{-1/4})$ (i.e. \eqref{eq:Lipsch} holds)) or $\frac{1}{X} \sum_{n \leq X} f(n)$ and $\frac{1}{Y} \sum_{n \leq Y} f(n)$ have the same sign. In the latter case we notice that~\eqref{eq:Lipcons} implies also (see also~\cite[Corollary 3]{GS03})
\[
\left|\left|\frac{1}{X} \sum_{n \leq X} f(n)\right| - \left|\frac{1}{Y} \sum_{n \leq Y} f(n)\right|\right| \ll \frac{1}{(\log X)^{1/4}},
\]
and~\eqref{eq:Lipsch} follows, since the averages have the same sign, so that the inner absolute values can be removed.
\end{proof}

We will actually need to apply the previous two lemmas for sums with the additional restriction $n \in \mathcal{S}$ where $\mathcal{S}$ is as in Section~\ref{se:KeyIdeas}. This can be done through the following immediate consequence of the inclusion-exclusion principle.
\begin{lemma}
\label{le:Sinclexcl}
Let $\mathcal{S}$ be as in Section~\ref{se:KeyIdeas}. For $\mathcal{J} \subseteq \{1, \dotsc, J\}$, let $g$ be the completely multiplicative function
\[
g_{\mathcal{J}}(p^j) = 
\begin{cases}
1 & \text{if $p \not \in \bigcup_{j \in \mathcal{J}} [P_j, Q_j]$} \\
0 & \text{otherwise}.
\end{cases}
\]
Then
\[
\sum_{\substack{ X \leq n \leq 2X \\n \in \mathcal{S}}} a_n = \sum_{\substack{ X \leq n \leq 2X}} a_n \prod_{j=1}^J (1-g_{\{j\}}(n)) = \sum_{\mathcal{J} \subseteq \{1, \dotsc, J\}} (-1)^{\# \mathcal{J}} \sum_{X \leq n \leq 2X} g_{\mathcal{J}}(n) a_n.
\]
\end{lemma}
\section{Mean and large value theorems for Dirichlet polynomials}

Let us first collect some standard mean and large value results for Dirichlet polynomials.
\begin{lemma}
\label{le:contMVT}
Let $A(s) = \sum_{n \leq N} a_n n^{-s}$. Then
$$
\int_{-T}^T |A(it)|^2 dt = (T+O(N)) \sum_{n \leq N} |a_n|^2
$$
\end{lemma}
\begin{proof}
See \cite[Theorem 9.1]{IwKo04}.
\end{proof}

For the rest of the paper we say that $\mathcal{T} \subseteq \mathbb{R}$ is well-spaced if $|t-r| \geq 1$ for all distinct $t, r \in \mathcal{T}$.

\begin{lemma}
\label{le:discMVT}
Let $A(s) = \sum_{n \leq N} a_n n^{-s}$, and let $\mathcal{T} \subset [-T, T]$ be a sequence of well-spaced points. Then
$$
\sum_{t \in \mathcal{T}} |A(it)|^2 \ll (T+N) \log 2N \sum_{n \leq N} |a_n|^2
$$
\end{lemma}
\begin{proof}
See \cite[Theorem 9.4]{IwKo04}.
\end{proof}

\begin{lemma}
\label{le:Rupest}
Let 
\[
P(s) = \sum_{P \leq p \leq 2P} \frac{a_p}{p^s} \quad \text{with $|a_p| \leq 1$.}
\]
Let $\mathcal{T} \subset [-T, T]$ be a sequence of well-spaced points such that $|P(1 + it)| \geq V^{-1}$ for every $t \in \mathcal{T}$. Then
\[
|\mathcal{T}| \ll T^{2 \frac{\log V}{\log P}} V^2 \exp\left(2 \frac{\log T}{\log P} \log \log T\right).
\]
\end{lemma}

\begin{proof}
Let $k = \lceil \log T / \log P \rceil$ and $$P(s)^k =: \sum_{P^k \leq n \leq (2P)^k} b(n) n^{-s}.$$
Notice that
\[
\begin{split}
&\sum_{P^k \leq n \leq (2P)^k} \left(\frac{b(n)}{n}\right)^2 \leq \sum_n \left(\sum_{\substack{p_1 \dotsm p_k = n \\ P \leq p_j \leq 2P}} \frac{1}{p_1 \dotsm p_k}\right)^2  \\
&\leq \frac{1}{P^k} \sum_{\substack{p_1 \dotsm p_k = q_1 \dotsm q_k  \\ P \leq p_j, q_j \leq 2P}} \frac{1}{p_1 \dotsm p_k} \leq \frac{1}{P^k} k! \Big ( \sum_{P \leq p \leq 2P} \frac{1}{p} \Big )^{k}.
\end{split}
\]

Hence by the previous lemma and Chebyschev's inequality
\[
\begin{split}
|\mathcal{T}| &\ll V^{2k} \cdot (T+ (2P)^k) \log (2P)^k \frac{1}{P^k} k! \Big( \sum_{P \leq p \leq 2P} \frac{1}{p} \Big )^{k} \\
&\ll T^{2\frac{\log V}{\log P}} V^2 5^k k!.
\end{split}
\]
\end{proof}

For sparse sets $\mathcal{T}$ one can use work of Hal\'asz to improve on the bound given for $\sum_{t \in \mathcal{T}} |A(it)|^2$ in Lemma~\ref{le:discMVT}. We will actually need two versions of Hal\'asz's inequality.
The first, stated below, works for arbitrary Dirichlet polynomials supported on integers.
The second, stated in Lemma \ref{le:Hallargevalprimes} requires the Dirichlet polynomial to be supported
on the primes, and is stronger in certain situations. 
Accordingly we call the first Lemma a ``Hal\'asz inequality for the integers'' and the second
a ``Hal\'asz inequality for the primes''. 
\begin{lemma}[Hal\'asz inequality for integers]
\label{le:Hallargevalint}
Let $A(s) = \sum_{n \leq N} a_n n^{-it}$ and let $\mathcal{T}$ be a sequence of well-spaced points. Then
$$
\sum_{t \in \mathcal{T}} |A(it)|^2 \ll (N + |\mathcal{T}| \sqrt{T}) \log 2T \sum_{n \leq N} |a_n|^2
$$
\end{lemma}
\begin{proof}
See \cite[Theorem 9.6]{IwKo04}.
\end{proof}
Let us now explain why we need a separate ``Hal\'asz inequality for the primes''. 
In all the mean and large value theorems presented so far, the term $N \sum_{n \leq N} |a_n|^2$ reflects the largest possible value of $|A(it)|^2$. However, when $n$ is supported on a thin sets such as primes, such a bound loses a logarithmic factor compared to the expected maximum (even when there is no $\log 2T$ or $\log 2N$ present). Our ``Hal\'asz inequality for the primes'' recovers this loss when $\mathcal{T}$ is very small, which is enough for us. The
proof relies on the duality principle, which we state below. 
\begin{lemma}[Duality principle]
\label{le:duality}
Let $\mathcal{X} = (x_{mn})$ be a complex matrix and $D \geq 0$. The following two statements are equivalent:
\begin{itemize}
\item For any complex numbers $a_n$
\[
\sum_{m}\left|\sum_{n} a_n x_{mn}\right|^2 \leq D \sum_{n} |a_n|^2;
\] 
\item For any complex numbers $b_m$
\[
\sum_{n}\left|\sum_{m} b_m x_{mn}\right|^2 \leq D \sum_{m} |b_m|^2.
\] 
\end{itemize}
\end{lemma}
\begin{proof}
See \cite[Chapter 7, Theorem 6, p. 134]{MontgomeryTenLecture}
\end{proof}

\begin{lemma}[Hal\'asz inequality for primes]
\label{le:Hallargevalprimes}
Let $P(s) = \sum_{P \leq p \leq 2P} a_p p^{-s}$ be a Dirichlet polynomial whose coefficients are supported on the primes and let $\mathcal{T} \subset [-T, T]$ be a sequence of well-spaced points. Then
$$
\sum_{t \in \mathcal{T}} |P(it)|^2 \ll \left(P + |\mathcal{T}| P \exp\left(-\frac{\log P}{(\log T)^{2/3+\varepsilon}}\right)(\log T)^2\right) \cdot \sum_{P \leq p \leq 2P} \frac{|a_p|^2}{\log P} .
$$
\end{lemma}
\begin{proof}
By the duality principle (Lemma~\ref{le:duality}) applied to $(p^{it})_{P \leq p \leq 2P, t \in \mathcal{T}}$ it is enough to prove that
$$
\sum_{P \leq p \leq 2P} \log p \left| \sum_{t \in \mathcal{T}} \eta_t p^{it} \right|^2 \ll \left(P + |\mathcal{T}| P \exp\left(-\frac{\log P}{(\log T)^{2/3+\varepsilon}}\right)(\log T)^2\right) \cdot \sum_{t \in \mathcal{T}} |\eta_t|^2
$$
for any complex numbers $\eta_t$. Opening the square, we see that
\begin{align*}
 \sum_{P \leq p \leq 2P} \log p \Big | \sum_{t \in \mathcal{T}} \eta_t p^{it} \Big |^2
& \leq 
\sum_{p^k} \log p \Big | \sum_{t \in \mathcal{T}} \eta_t p^{kit} \Big |^2 f \Big ( \frac{p^k}{P} \Big ) \\ 
& \leq \sum_{t,t'\in\mathcal{T}} |\eta_t \eta_{t'}| \Big | \sum_{p^k} \log p \cdot p^{ki(t - t')} f \Big ( \frac{p^k}{P} \Big )
\Big |
\end{align*}
where $f(x)$ is a smooth compactly supported function such that
$f(x) = 1$ for $1 \leq x \leq 2$ and $f$ decays to zero outside of the interval $[1,2]$.  
Let $\widetilde{f}$ denote the Mellin transform of $f$. Then 
$\widetilde{f}(x + iy) \ll_{A,B} (1 + |y|)^{-B}$ uniformly in $|x| \leq A$.  In addition,
\begin{align} \label{eq:sidecomp} 
\sum_{n} \Lambda(n) n^{i t} & f \Big ( \frac{n}{P} \Big )  =
-\frac{1}{2\pi i} \int_{2 - i \infty}^{2 + i \infty} 
\widetilde{f}(s) \frac{\zeta'}{\zeta}(s - it)
\frac{P^s}{s} ds 
\end{align}
We truncate the integral at $|t| = T$, making a negligible error of $O_A(T^{-A})$. In the remaining integral, we shift the contour to $\sigma = 1 - c ( \log T)^{-2/3 + \varepsilon}$, staying in the zero-free region of the $\zeta$-function, and use the following bound there (see \cite[formula (1.52)]{Ivic03})
$$
\frac{\zeta'}{\zeta}(\sigma + it) = \sum_{\substack{\varrho = \beta+i\gamma \\ |t-\gamma| < 1}} \frac{1}{\sigma+it-\varrho} + O(\log(|t|+2)) \ll (\log T)^{1 + 2/3 + \varepsilon}
$$
One readily checks this bound by noticing that there are $O(\log T)$ zeros in the sum and they are $\gg (\log T)^{-2/3 + \varepsilon}$ away from the contour. It follows that (\ref{eq:sidecomp}) is equal to 
\begin{align*}
& \frac{\widetilde{f}(1 + it)}{1 + it} \cdot P^{1 + it}
+ O\left(P \exp\left(-\frac{\log P}{(\log T)^{2/3+\varepsilon}}\right)(\log T)^2\right)
\end{align*}
Combining the above observations and using the inequality $|\eta_t \eta_{t'}| \leq |\eta_t|^2 + |\eta_{t'}|^2$ we obtain
\[
\begin{split}
&\sum_{P \leq p \leq 2P} \log p \left| \sum_{t \in \mathcal{T}} \eta_t p^{it} \right|^2 \\
&\ll \sum_{t,t'\in\mathcal{T}} |\eta_t \eta_{t'}| \Big | \sum_{p^k} \log p \cdot p^{ki(t - t')} f \Big ( \frac{p^k}{P} \Big )
\Big | \\
& \ll \sum_{t,t'\in\mathcal{T}} (|\eta_t|^2+|\eta_{t'}|^2) \left(\left| \frac{\widetilde{f}(1 + i(t-t'))}{1 + i(t-t')}\right| \cdot P
+ P \exp\left(-\frac{\log P}{(\log T)^{2/3+\varepsilon}}\right)(\log T)^2\right) \\
&
\ll \Big (P + |\mathcal{T}| P \exp\left(-\frac{\log P}{(\log T)^{2/3+\varepsilon}}\right)(\log T)^2 \Big ) \cdot \sum_{t \in \mathcal{T}} |\eta_t|^2
\end{split}
\]
since $\sum_{t \in \mathcal{T}} |\widetilde{f}(1 - i(t - t'))| = O(1)$. 
\end{proof}
\begin{remark}
On the Riemann Hypothesis one can replace $P \exp(- \log P / (\log T)^{2/3 + \varepsilon}) (\log T)^2$ in the above lemma by $P^{1/2} \log P \log T$. 
\end{remark}

\section{Decomposition of Dirichlet polynomials}
In this section we prove a technical version of the Buchstab decomposition~(\ref{eq:Buchstab}). We are grateful to Terry Tao
for pointing out that our ``Buchstab decomposition'' is a variant of Ramar\'e's identity \cite[Section 17.3]{Opera}.

\begin{lemma} \label{lem:decomp}
  Let $H \geq 1$ and $Q \geq P \geq 1$. Let $a_m, b_m$ and $c_p$ be bounded sequences such that $a_{mp} = b_m c_p$ whenever $p \nmid m$ and $P \leq p \leq Q$. Let
\begin{align*}
Q_{v,H}(s) & = \sum_{\substack{P \leq p \leq Q \\ e^{v/H} \leq p \leq e^{(v + 1)/H}}}
\frac{c_p}{p^s} \quad \text{and} \\
R_{v,H}(s) & = \sum_{\substack{X e^{-v/H} \leq m \leq 2X e^{-v/H}}} \frac{b_m}{m^s} \cdot \frac{1}{\#\{P \leq q \leq Q: q | m,  q \in \mathbb{P}\} + 1}
\end{align*}
and let $\mathcal{T} \subseteq [-T, T]$.
Then, 
\begin{align*}
\int_{\mathcal{T}} & \Big | \sum_{\substack{ X \leq n \leq 2X}}
\frac{a_n}{n^{1+it}} \Big |^2 dt \ll H \log \Big ( \frac{Q}{P} \Big ) \times \sum_{j \in \mathcal{I}} \int_{\mathcal{T}} \Big | Q_{j,H}(1 + it) R_{j,H}(1 + it)|^2 dt \\ & + \frac{T + X}{X} \Bigg (\frac{1}{H} + \frac{1}{P} + \sum_{\substack{X \leq n \leq 2X \\ (n, \prod_{P \leq p \leq Q} p) = 1}} \frac{|a_n|^2}{n} \Bigg )
\end{align*}
where $\mathcal{I}$ is the interval $\lfloor H \log P \rfloor \leq j \leq  H \log Q$.
\end{lemma}

\begin{proof}
Let us write $s=1+it$ and notice that 
\begin{align} \label{firstequ}
\sum_{\substack{ X \leq n \leq 2X}} \frac{a_n}{n^s} =
\sum_{P \leq p \leq Q} & \sum_{\substack{X/p \leq m \leq 2X/p}} \frac{a_{pm}}{(pm)^s} \cdot \frac{1}{\# \{P \leq q \leq Q: q | m, q \in \mathbb{P} \} + \mathbf{1}_{(p,m) = 1}}  + \sum_{\substack{X \leq n \leq 2X \\ (n,\mathcal{P}) = 1}} \frac{a_n}{n^s}
\end{align}
where $\mathcal{P} = \prod_{P \leq p \leq Q} p $
and $\mathbf{1}_{(p,m) = 1}$ is the indicator function of $(p,m) = 1$. 
Notice that when $p \nmid m$, we can replace $a_{pm}$ by $b_mc_p$. Let also $\omega(n;P,Q) = \#\{P \leq p \leq Q: p | n\}$. This allows us to rewrite the first summand as
\begin{align*}
  \sum_{P \leq p \leq Q} & \frac{c_p}{p^s} \sum_{\substack{X/p \leq m \leq 2X/p}} \frac{b_m}{m^s} \cdot \frac{1}{\omega(m;P,Q) + 1} \\ & + \sum_{P\leq p \leq Q} \sum_{\substack{X/p \leq m \leq 2X/p \\ p \mid m}}
\Big (   \frac{a_{p m}}{\omega(m; P, Q)} \cdot \frac{1}{(p m)^{s}} - \frac{b_{m}c_p}{(p m)^s} \cdot \frac{1}{\omega (m;P,Q) + 1} \Big ). 
\end{align*}
We split the first sum further into dyadic ranges getting that it is
$$
\sum_{j \in \mathcal{I} } \ \sum_{\substack{e^{j/H} \leq p < e^{(j+1)/H} \\ P \leq p \leq Q}}
\frac{c_p}{p^s} \ \sum_{\substack{X e^{-(j+1)/H} \leq m \leq 2X e^{-j/H} \\ X \leq m p \leq 2X}} \frac{b_m}{m^s} \cdot \frac{1}{\omega(m; P, Q) + 1}
$$ 
We remove the condition $X \leq mp \leq 2X$ overcounting at most by the integers $mp$ in the ranges $[X e^{-1/H}, X]$ and $[2X, 2X e^{1/H}]$. Similarly, removing numbers with $Xe^{-(j+1)/H} \leq m \leq Xe^{-j/H}$ we undercount at most by integers $mp$ in the range $[Xe^{-1/H}, Xe^{1/H}]$. Therefore we can, for some bounded $d_m$, rewrite (\ref{firstequ}) as
\begin{align*}
\sum_{j \in \mathcal{I}} & Q_{j, H}(s) R_{j, H}(s)+  \sum_{\substack{X e^{-1/H} \leq m \leq Xe^{1/H}}} \frac{d_m}{m^s} + \sum_{\substack{2X \leq m \leq 2X e^{1/H}}} \frac{d_m}{m^s} \\
                         & + \sum_{P \leq p \leq Q} \sum_{\substack{X/p^2 \leq m \leq 2 X/p^2}}
                           \Big ( \frac{a_{p^2 m}}{\omega(mp; P, Q)} \cdot \frac{1}{(p^2 m)^s} - \frac{c_p b_{p m}}{\omega(m p; P, Q) + 1} \cdot \frac{1}{(p^2 m)^s} \Big ) 
                           + \sum_{\substack{X \leq n \leq 2X \\ (n,\mathcal{P}) = 1}} \frac{a_n}{n^s}
\end{align*}
We square this, integrate over $\mathcal{T}$ and then apply Cauchy-Schwarz on the first sum over $j$ and the mean-value theorem (Lemma~\ref{le:contMVT}) on the remaining sums. This gives the result since it is easily seen that the later mean-values are bounded by the stated quantities. 
\end{proof}

\section{Moment computation}
In this section we prove a lemma which allows us to compute the second moment of the Dirichlet polynomial in~(\ref{eq:thedirpoly}). Let us first introduce some
relevant notation. 
Let $Y_1, Y_2 \geq 1$, and consider,
$$
Q(s) = \sum_{Y_1 \leq p \leq 2Y_1} \frac{c_p}{p^s} \quad \text{and} \quad 
A(s) = \sum_{\substack{X/Y_2 \leq m \leq 2X/Y_2}} \frac{a_m}{m^s}
$$
with coefficients $|a_m|, |c_p| \leq 1$. 
\begin{lemma}
\label{le:moment}
Let $\ell = \lceil  \frac{\log Y_2}{\log Y_1} \rceil$. Then
$$
\int_{-T}^{T} |Q(1 + it)^{\ell} \cdot A(1 + it)|^2 dt \ll \Big ( \frac{T}{X} + 2^\ell Y_1 \Big ) \cdot (\ell + 1)!^2
$$
\end{lemma}
\begin{proof}
The coefficients of the Dirichlet polynomial $Q(s)^{\ell} A(s)$ are supported on the interval
$$
[Y_1^\ell \cdot X/Y_2, (2Y_1)^\ell\cdot 2 X/Y_2] \subseteq [X, 2^{\ell+1}Y_1 X]
$$
Using the mean-value theorem for Dirichlet polynomials (Lemma~\ref{le:contMVT}) we see that
$$
\int_{-T}^{T} |Q(1 + it)^{\ell} \cdot A(1 + it)|^2 dt \ll (T + 2^{\ell} Y_1 X) \sum_{X \leq n \leq 2^{\ell+1} Y_1 X} \frac{1}{n^2} \cdot
\Bigg ( \sum_{\substack{n = m p_1 \ldots p_{\ell} \\ Y_1 \leq p_1, \ldots, p_{\ell} \leq 2Y_1 \\ X/Y_2 \leq m \leq 2X/Y_2}} 1 \Bigg )^2.
$$
Here
\[
\sum_{\substack{n = m p_1 \ldots p_{\ell} \\ Y_1 \leq p_1, \ldots, p_{\ell} \leq 2Y_1 \\ X/Y_2 \leq m \leq 2X/Y_2}} 1 \leq \ell! \cdot \sum_{\substack{n = mr \\ p \mid r \implies Y_1 \leq p \leq 2Y_1}} 1 =: \ell! \cdot g(n),
\]
say, where $g$ is multiplicative and
\[
g(p^k) = \begin{cases}
(k+1) & \text{if $Y_1 \leq p \leq 2Y_1$;} \\
1 & \text{otherwise.}
\end{cases}
\]
With this notation
\begin{equation}
\label{eq:Momboundg(n)}
\int_{-T}^{T} |Q(1 + it)^{\ell} \cdot A(1 + it)|^2 \ll (T + 2^{\ell} Y_1 X) \ell!^2 \sum_{X \leq n \leq 2^{\ell+1} Y_1 X} \frac{g(n)^2}{n^2}.
\end{equation}
By Shiu's bound~\cite[Theorem 1]{Shiu80} for sums of positive-valued multiplicative functions we have, for any $Y \geq 2$,
\begin{equation}
\label{eq:Shiucons}
\sum_{Y \leq n \leq 2Y} g(n)^2 \ll Y\prod_{p \leq Y} \left(1+\frac{|g(p)|^2-1}{p}\right) \ll Y.
\end{equation}
The claim follows by splitting the sum over $n$ in~\eqref{eq:Momboundg(n)} into sums over dyadic intervals and applying~\eqref{eq:Shiucons} to each of them.
\end{proof}

\section{Parseval bound}
The following lemma shows that the behavior of a multiplicative function in almost all very short intervals can be approximated by its behavior on a long interval if the mean square of the corresponding Dirichlet polynomial is small. This is in the spirit of previous work on primes in almost all intervals, see for instance~\cite[Lemma 9.3]{Harman06}.
\begin{lemma}
\label{lem:shortintsumtoDirpols}
Let $|a_m| \leq 1$. Assume $1 \leq h_1 \leq h_2 = X / (\log X)^{1/5}$. Consider, for $X \leq x \leq 2 X$, 
$$
S_j(x) = \sum_{\substack{x \leq m \leq x+h_j}} a_m \quad \text{and write} \quad A(s) := \sum_{\substack{X \leq m \leq 4X}} \frac{a_m}{m^s}.
$$
Then
\[
\begin{split}
&\frac{1}{X} \int_{X}^{2X} \left|\frac{1}{h_1} S_1(x) - \frac{1}{h_2}S_2(x)\right|^2 dx \\
&\ll \frac{1}{(\log X)^{2/15}} + \int_{1+i(\log X)^{1/15}}^{1+iX/h_1} \left|A(s)\right|^2 |ds| + \max_{T \geq X/h_1} \frac{X/h_1}{T} \int_{1+iT}^{1+i2T} \left|A(s)\right|^2 |ds|.
\end{split}
\]
\end{lemma}
\begin{proof}
By Perron's formula
\[
\begin{split}
S_j(x) & 
 =
\frac{1}{2\pi i} \int_{1 - i\infty}^{1 + i\infty}  A(s) \frac{(x+h_j)^s-x^s}{s} ds. 
\end{split}
\]
Let us split the integral in $S_j(x)$ into two parts $U_j(x)$ and $V_j(x)$ according to whether $|t| \leq T_0 := (\log X)^{1/15}$ or not. In $U_j(x)$ we write 
\[
\frac{(x+h_j)^s-x^s}{s} = x^s\frac{\left(1+\frac{h_j}{x}\right)^s-1}{s} = x^s\left(\frac{h_j}{x} + O\left(T_0\left(\frac{h_j}{X}\right)^2\right)\right),
\]
and get
\[
U_j(x) = \frac{h_j}{x} \cdot \frac{1}{2 \pi i} \int_{1-iT_0}^{1+iT_0} A(s) x^{s} ds + O\left(T_0^2 \cdot x \left(\frac{h_j}{X}\right)^2 \right),
\]
so that
\[
\frac{1}{h_1}U_1(x) - \frac{1}{h_2} U_2(x) \ll T_0^2 x \frac{h_2}{X^2}  \ll \frac{1}{(\log X)^{1/15}}.
\]

Hence it is enough to consider, for $j = 1, 2$,
\[
\begin{split}
&\frac{1}{X}\int_{X}^{2X} \left(\frac{|V_j(x)|}{h_j}\right)^2 dx \ll \frac{1}{h_j^2 X} \int_{X}^{2X} \left|\int_{1 + iT_0}^{1 + i\infty}  A(s) \frac{(x+h_j)^s-x^s}{s}ds\right|^2 dx.
\end{split}
\]
We would like to add a smoothing, take out a factor $x^s$, expand the square, exchange the order of integration and integrate over $x$. However, the term $(x+h_j)^s$ prevents us from doing this and we overcome this problem in a similar way to \cite[Page 25]{SaVa77}. We write
\[
\begin{split}
&\frac{(x+h_j)^s - x^s}{s} = \frac{1}{2h_j} \left(\int_{h_j}^{3h_j} \frac{(x+w)^s - x^s}{s} dw - \int_{h_j}^{3h_j} \frac{(x+w)^s - (x+h_j)^s}{s} dw\right) \\
&= \frac{x}{2h_j} \int_{h_j/x}^{3h_j/x} x^s\frac{(1+u)^s - 1}{s} du - \frac{x+h_j}{2h_j} \int_{0}^{2h_j/(x+h_j)} (x+h_j)^s \frac{(1+u)^s - 1}{s} du.
\end{split}
\]
where we have substituted $w = x \cdot u$ in the first integral and $w = h_j+ (x+h_j)u$ in the second integral. Let us only study the first summand, the second one being handled completely similarly. Thus we assume that
\[
\begin{split}
\frac{1}{X}\int_X^{2X}\left(\frac{|V_j(x)|}{h_j}\right)^2 dx &\ll \frac{X}{h_j^4} \int_X^{2X}\left|\int_{h_j/x}^{3h_j/x} \int_{1+iT_0}^{1+i\infty} A(s) x^s\frac{(1+u)^s - 1}{s} ds du \right|^2 dx \\
&\ll \frac{1}{h_j^3} \int_{h_j/(2X)}^{3h_j/X} \int_{X}^{2X}\left|\int_{1+iT_0}^{1+i\infty} A(s) x^s\frac{(1+u)^s - 1}{s} ds \right|^2dx du \\
&\ll \frac{1}{h_j^2X} \int_{X}^{2X}\left|\int_{1+iT_0}^{1+i\infty} A(s) x^s\frac{(1+u)^s - 1}{s} ds \right|^2dx
\end{split}
\]
for some $u \ll h_j/X$.

Let us introduce a smooth function $g(x)$ supported on $[1/2, 4]$ and
equal to $1$ on $[1,2]$. We obtain
%
\begin{align*}
&\frac{1}{X}\int_X^{2X}\left(\frac{|V_j(x)|}{h_j}\right)^2 dx \ll \frac{1}{h_j^2X} \int g\Big ( \frac{x}{X} \Big ) \left|\int_{1+iT_0}^{1+i\infty} A(s) x^s\frac{(1+u)^s - 1}{s} ds \right|^2dx \\
&\leq \frac{1}{h_j^2 X} \int_{1+iT_0}^{1+i\infty} \int_{1+iT_0}^{1+i\infty} \left|A(s_1) A(s_2) \frac{(1+u)^{s_1}-1}{s_1} \frac{(1+u)^{s_2}-1}{s_2}\right| \left| \int g \Big ( \frac{x}{X} \Big ) x^{s_1+\overline{s_2}} dx \right| |ds_1 ds_2| \\
&\ll  \frac{1}{h_j^2X}\int_{1+iT_0}^{1+i\infty} \int_{1+iT_0}^{1+i\infty} |A(s_1) A(s_2)| \min\left\{\frac{h_j}{X}, \frac{1}{|t_1|}\right\} \min\left\{\frac{h_j}{X}, \frac{1}{|t_2|}\right\} \frac{X^3}{|t_1-t_2|^2+1} |ds_1 ds_2| \\
&\ll  \frac{X^2}{h_j^2} \int_{1+iT_0}^{1+i\infty} \int_{1+iT_0}^{1+i\infty} \frac{|A(s_1)|^2 \min\{(h_j/X)^2, |t_1|^{-2}\}+ |A(s_2)|^2\min\{(h_j/X)^2, |t_2|^{-2}\}}{|t_1-t_2|^2+1} |ds_1 ds_2| \\
&\ll \int_{1+iT_0}^{1+iX/h_j} |A(s)|^2 |ds| + \frac{X^2}{h_j^2} \int_{1+iX/h_j}^{1+i\infty} \frac{|A(s)|^2}{|t|^2} |ds|.
\end{align*}
The second summand is
\begin{equation}
\label{eq:intsplit}
\ll \frac{X^2}{h_j^2} \int_{1+iX/(2h_j)}^{1+i\infty} \frac{1}{T^3} \int_{1+iT}^{1+i2T} |A(s)|^2 |ds| dT \ll \frac{X^2}{h_j^2} \cdot \frac{1}{X/h_j} \max_{T \geq X/(2h_j)} \frac{1}{T} \int_{1+iT}^{1+i2T} |A(s)|^2 |ds| 
\end{equation}
so that
\[
\begin{split}
&\frac{1}{X}\int_{X}^{2X} \left(\frac{|V_j(x)|}{h_j}\right)^2 dx \ll \int_{1+iT_0}^{1+iX/h_j} |A(s)|^2 |ds| + \frac{X}{h_j}\max_{T \geq X/h_j} \frac{1}{T} \int_{1+iT}^{1+i2T} |A(s)|^2 |ds|.
\end{split}
\]
Since $h_2 \geq h_1$ the expression on the right hand side with $j=2$ is always smaller than the same expression with $j=1$, and the claim follows.
\end{proof}

\section{The main proposition} 
\label{sec:mainProp}
By Lemma~\ref{lem:shortintsumtoDirpols}, Theorem~\ref{th:ThminS} will essentially follow from the following proposition. 
\begin{proposition}
\label{prop:MainProp}
Let $f : \mathbb{N} \rightarrow [-1,1]$ be a multiplicative
function. Let $\mathcal{S}$ be a set of integers as defined in Section \ref{se:KeyIdeas}. 
Let
\[
F(s) = \sum_{\substack{X \leq n \leq 2X \\ n \in \mathcal{S}}} \frac{f(n)}{n^s}.
\]
Then, for any $T$,
\[
\int_{(\log X)^{1/15}}^T \left|F(1+it)\right|^2 dt \ll \left(\frac{T}{X/Q_1} + 1\right) \left(\frac{(\log Q_1)^{1/3}}{P_1^{1/6-\eta}} + \frac{1}{(\log X)^{1/50}}\right).
\]
\end{proposition}
\begin{remark}
The ``trivial bound'' for $\int_{0}^{T} |F(1 + it)|^2 dt$, obtained by applying a standard mean-value theorem (Lemma~\ref{le:contMVT}), is $T / X + 1$. 
\end{remark}
\begin{proof}
Since the mean value theorem gives the bound $O(\frac{T}{X}+1)$, we can assume $T \leq X$.

Pick a sequence $\alpha_j$ for $1 \leq j \leq J$ with 
\begin{equation}
\label{eq:alphajdef}
\alpha_j = \frac{1}{4} - \eta\left(1+ \frac{1}{2j}\right),
\end{equation}
where $\eta \in (0, 1/6)$ is such that~\eqref{eq:PjQjnottoofar} and~\eqref{eq:PjQjnottooclose} hold. Notice that 
\[
\frac{1}{4}-\frac{3}{2}\eta = \alpha_1 \leq \alpha_2 \leq \dotsc \leq \alpha_J \leq \frac{1}{4}-\eta.
\]
We now split into several cases. Let
\[
Q_{v, H_j}(s) := \sum_{\substack{P_j \leq q \leq Q_j \\ e^{v/H_j} \leq q \leq e^{(v+1)/H_j}}} \frac{f(q)}{q^s}, \quad \text{where} \quad H_j := j^2 \frac{P_1^{1/6-\eta}}{(\log Q_1)^{1/3}}. 
\]
Notice that this can be non-zero only when 
\[
v \in \mathcal{I}_j := \{v: \lfloor H_j \log P_j \rfloor \leq v \leq H_j \log Q_j\}
\]
We write
\[
[T_0, T] = \bigcup_{j = 1}^J \mathcal{T}_j \cup \mathcal{U} \ , \ T_0 = (\log X)^{1/15}
\]
as a disjoint union where $t \in \mathcal{T}_j$ when $j$ is the smallest index such that
\begin{equation}
\label{eq:T_jdefcond}
\text{for all } v \in \mathcal{I}_j: |Q_{v,H_j}(1+it)| \leq e^{-\alpha_j v / H_j}
\end{equation}
and $t \in \mathcal{U}$ if this does not hold for any $j$.

Let us first consider the integrals over the sets $\mathcal{T}_j$. Let
$$
R_{v,H_j}(s) = \sum_{\substack{X e^{-v/H_j} \leq m \leq 2 X e^{-v/H_j} \\ m \in \mathcal{S}_j}} \frac{f(m)}{m^s} \cdot \frac{1}{\#\{P_j \leq p \leq Q_j:
p|m \} +1}
$$
where $\mathcal{S}_j$ is the set of those integers
which have at least one prime factor in every
interval $[P_i,Q_i]$ with $i \neq j$ and $i \leq J$ (and possibly but not necessarily some prime factors in $[P_j, Q_j]$).
Using Lemma \ref{lem:decomp} with $H = H_j$, $P = P_j, Q = Q_j$ and $a_m = f(m) \mathbf{1}_{\mathcal{S}}$, $c_p = f(p)$, $b_m = f(m) \mathbf{1}_{\mathcal{S}_j}$
(where $\mathbf{1}_{A}$ is the indicator function of the set $A$), 
we see that 
\begin{align*}
\int_{\mathcal{T}_j}  | F(1 + it)|^2 dt \ll H_j \cdot \log Q_j \sum_{v \in \mathcal{I}_j} \int_{\mathcal{T}_j}
|Q_{v,H_j}(1 + it)R_{v,H_j}(1 + it)|^2 dt + \frac{1}{H_j} + \frac{1}{P_j}.
\end{align*}
Here the second and third terms contribute in total to integrals over all $\mathcal{T}_j$
\[
\ll \sum_{j = 1}^J \left(\frac{1}{H_j} + \frac{1}{P_j}\right) \ll \frac{(\log Q_1)^{1/3}}{P_1^{1/6-\eta}}.
\]
since $P_j \geq P_1^{j^2}$ by~\eqref{eq:PjQjnottooclose}. 
We can thus concentrate, for $1 \leq j \leq J$, on bounding
\begin{align} \label{vdef}
E_j := H_j \log Q_j \cdot \sum_{v \in \mathcal{I}_j} \int_{\mathcal{T}_j} |Q_{v, H_j}(1+it) R_{v, H_j}(1+it)|^2 dt.
\end{align}
By the definition of the set $\mathcal{T}_j$ we have $|Q_{v, H_j}(1 + it)| \leq e^{-\alpha_j v / H_j}$ for $t \in \mathcal{T}_j$. Therefore, for $1 \leq j \leq J$,
\begin{equation}
\label{eq:intboundPjout}
E_j \ll H_j \log Q_j
 \cdot \sum_{v \in \mathcal{I}_j} e^{-2\alpha_j v / H_j} \int_{\mathcal{T}_j} |R_{v, H_j}(1+it)|^2 dt.
\end{equation}
Recalling that $[T_0, T] = \mathcal{T}_1 \cup \mathcal{T}_2 \cup \ldots \cup \mathcal{T}_J \cup \mathcal{U}$ 
(with $T_0 = (\log X)^{1/15}$) we see that
\begin{equation} \label{eq:boundplan}
\int_{T_0}^{T} |F(1 + it)|^2 dt \ll E_1 + E_2 + \ldots + E_J + \int_{\mathcal{U}} |F(1 + it)|^2 dt
\end{equation}
We will now proceed as follows: In section 8.1 we bound $E_1$, in section 8.2 we bound $E_i$ with $2 \leq i \leq J$, and
finally in section 8.3 we obtain a bound for $\int_{\mathcal{U}} |F(1 + it)|^2 dt$. 
\subsection{Bounding $E_1$}
If $j = 1$, then by the mean-value theorem (Lemma~\ref{le:contMVT}), we get
\begin{align*}
E_1 &\ll H_1 \log Q_1 \cdot \sum_{v \in \mathcal{I}_1} e^{-2\alpha_1 v/H_1} \cdot \Big ( T + \frac{X}{e^{v/H_1}} \Big ) \frac{1}{X/e^{v/H_1}} \\
&\ll H_1 \log Q_1 \cdot P_1^{-2\alpha_1} \frac{1}{1-e^{-2\alpha_1/H_1}} \cdot \Big ( \frac{T}{X/Q_1} + 1 \Big ) \\
&\ll H_1^2 \log Q_1 \cdot P_1^{-1/2+3\eta} \Big ( \frac{T}{X/Q_1} + 1 \Big ) \ll \Big ( \frac{T}{X/Q_1} + 1 \Big ) \frac{(\log Q_1)^{1/3}}{P_1^{1/6-\eta}}
\end{align*}
by the choice of $H_1$.
\subsection{Bounding $E_j$ with $2 \leq j \leq J$}
Now suppose that $2 \leq j \leq J$.  In this case we split further
\[
\mathcal{T}_j = \bigcup_{r \in \mathcal{I}_{j-1}} \mathcal{T}_{j, r},
\]
where 
$$
\mathcal{T}_{j, r} = \{t \in \mathcal{T}_j \colon |Q_{r,H_{j-1}}(1 + it)| > e^{-\alpha_{j-1} r/ H_{j-1}} \}
$$
Note that this is indeed a splitting, since, by the definition of $\mathcal{T}_j$, for any $t \in \mathcal{T}_j$ there will be an index $r \in \mathcal{I}_{j-1}$ such that $|Q_{r, H_{j-1}}(1 + it)| > e^{-\alpha_{j-1} r / H_{j-1}}$. Therefore, for some $v = v(j) \in \mathcal{I}_j$ and $r = r(j) \in \mathcal{I}_{j-1}$,
\begin{equation}
\label{eq:intboundrsplit}
E_j \ll H_j \log Q_j \cdot \# \mathcal{I}_j \cdot \# \mathcal{I}_{j-1} \cdot e^{-2 \alpha_j v / H_j}
\times \int_{\mathcal{T}_{j,r}} |R_{v, H_j}(1 + it)|^2 dt
\end{equation}
On $\mathcal{T}_{j,r}$ we have $|Q_{r,H_{j-1}}(1 + it)| > e^{-\alpha_{j-1} r/H_{j-1}}$. 
Therefore, for any $\ell_{j,r} \geq 1$, multiplying by the term $(|Q_{r,H_{j-1}}(1 + it)| e^{\alpha_{j-1} r/H_{j-1}})^{2\ell_{j,r}} \geq 1$, we can bound this further as
\begin{align*}
E_{j} \ll &(H_j \log Q_j)^3 \cdot e^{-2\alpha_j v/ H_j} \times \\ & \times  \exp \Big (2 \ell_{j,r} \cdot \alpha_{j-1} r  / H_{j-1} \Big )
\int_{\mathcal{T}_{j,r}} |Q_{r, H_{j-1}}(1 + it)^{\ell_{j,r}} R_{v,H_j}(1 + it)|^2 dt.
\end{align*}
Choosing
\[
\ell_{j,r} = \left\lceil \frac{v/ H_{j}}{r/H_{j-1}}\right \rceil \leq \frac{H_{j-1}}{r} \cdot \frac{v}{H_j} + 1,
\]
we get
\begin{align*}
E_j &\ll H_j^3 (\log Q_j)^3 \cdot \exp \Big ( 2v (\alpha_{j-1} - \alpha_{j}) / H_{j} + 2 \alpha_{j-1} r / H_{j-1} \Big ) \\
& \quad \cdot\int_{-T}^{T} |Q_{r,H_{j-1}}(1 + it)^{\ell_{j,r}} R_{v,H_j}(1 + it)|^2 dt.
\end{align*}
Now we are in the position to use Lemma~\ref{le:moment} which gives
\begin{align*}
\int_{-T}^{T}  |Q_{r,H_{j-1}}(1 + it)^{\ell_{j,r}}  R_{v,H_j}(1 + it)|^2  dt 
&\ll \left ( \frac{T}{X} + 2^{\ell_{j, r}}e^{r/H_{j-1}} \right ) \cdot (\ell_{j,r} + 1)!^2\\
&\ll \left( \frac{T}{X} + Q_{j-1} \right) \exp \left (2 \ell_{j,r} \log \ell_{j,r} \right ) 
\end{align*}
Here by the mean value theorem and the definition of $\ell_{j,r}$
\[
\begin{split}
\ell_{j,r} \log \ell_{j,r} &\leq \frac{v/ H_{j}}{r/H_{j-1}} \log \frac{v/ H_{j}}{r/H_{j-1}} + \log \log Q_j + 1 \\
&\leq \frac{v}{H_j} \cdot \frac{\log \log Q_j}{\log P_{j-1}-1} + \log \log Q_j + 1,
\end{split}
\]
so that
\begin{align*}
&\int_{-T}^{T}  |Q_{r,H_{j-1}}(1 + it)^{\ell_{j,r}}  R_{v,H_j}(1 + it)|^2  dt \\
& \ll \left( \frac{T}{X} + 1 \right) Q_{j-1} (\log Q_j)^2 \exp \left (\frac{v}{H_j} \cdot \frac{2\log \log Q_j}{\log P_{j-1}-1} \right ) \\
& \ll \left( \frac{T}{X} + 1 \right ) Q_{j-1} (\log Q_j)^2 \exp \left (\frac{\eta}{2j^2} \cdot \frac{v}{H_j} \right )
\end{align*}
by \eqref{eq:PjQjnottoofar}. Note that~\eqref{eq:PjQjnottoofar} also implies
\[
\log \log Q_j \leq \frac{1}{24}\log P_{j-1} \leq \log Q_{j-1}^{1/24} \implies \log Q_j \leq Q_{j-1}^{1/24},
\]
so that
\[
\begin{split}
H_j^3 (\log Q_j)^5 Q_{j-1} \exp(2\alpha_{j-1} r / H_{j-1})  &\ll H_j^3 (\log Q_j)^5 Q_{j-1}^2 \\
&\ll H_j^3 Q_{j-1}^{5/2} \ll j^6 P_1^{1/2} Q_{j-1}^{5/2} \ll j^6 Q_{j-1}^{3}.
\end{split}
\]
Therefore we end up with the bound
\begin{align*}
E_j & \ll \left( \frac{T}{X} + 1 \right ) j^6 Q_{j-1}^3 \exp \left ( \frac{2 v}{H_j} \left (\alpha_{j-1} - \alpha_{j} + \frac{\eta}{4j^2} \right ) \right) \\
& \ll \left( \frac{T}{X} + 1 \right ) j^6 Q_{j-1}^3 \exp \left ( - \frac{\eta}{2j^2} \log P_j \right ) \\
&\ll \left( \frac{T}{X} + 1 \right ) \frac{1}{j^2 Q_{j-1}} \ll \Big( \frac{T}{X} + 1 \Big ) \frac{1}{j^2 P_1}
\end{align*}
by~\eqref{eq:alphajdef} and~(\ref{eq:PjQjnottooclose}).

\subsection{Bounding $\int_{\mathcal{U}} |F(1 + it)|^2 dt$} Let us now bound the integral
$$
\int_{\mathcal{U}} |F(1 + it)|^2 dt.
$$
We again apply Lemma~\ref{lem:decomp}, this time with $a_m = b_m = f(m)\mathbf{1}_{\mathcal{S}}(m)$, $c_p = f(p)$ and $P = \exp((\log X)^{1-1/48}), Q = \exp(\log X/(\log \log X))$ and $H = (\log X)^{1/48}$ to see that, for some $v \in [\lfloor H \log P \rfloor, H \log Q]$, the integral is bounded by
$$
H^2(\log X)^{2} \int_{\mathcal{U}} |Q_{v,H}(1 + it)
R_{v,H}(1 + it)|^2 dt + \left(\frac{T}{X}+1\right)\left(\frac{1}{H}+\frac{1}{P}+ \frac{\log P}{\log Q}\right),
$$
where
$$
Q_{v,H}(s) = \sum_{e^{v/H} \leq p \leq e^{(v+1)/H}} \frac{f(p)}{p^{s}}
$$
and 
\[
R_{v,H}(s) = \sum_{\substack{X e^{-v/H} \leq m \leq 2X e^{-v/H} \\ m \in \mathcal{S}}} \frac{f(m)}{m^s} \cdot \frac{1}{\# \{p \in [P, Q] \colon p \mid m\} + 1}.
\]
We then find a well-spaced set $\mathcal{T} \subseteq \mathcal{U}$ such that
$$
\int_{\mathcal{U}} |Q_{v,H}(1 + it)R_{v,H}(1 + it)|^2 dt
\leq 2\sum_{t \in \mathcal{T}} |Q_{v,H}(1 + it)|^2 \cdot |R_{v,H}(1 + it)|^2.
$$
By definition of $J$ and~\eqref{eq:PjQjnottoofar}, we know that $Q_J \leq \exp((\log X)^{1/2})$ and
\[
\log P_J \geq \frac{4j^2}{\eta} \cdot \log \log Q_{J+1} \geq \frac{4j^2}{\eta} \cdot \log (\log X)^{1/2} \implies P_J \geq (\log X)^{2/\eta}.
\]
Now, by definition of $\mathcal{U}$, for each $t \in \mathcal{T}$ there is $v \in \mathcal{I}_J$ such that $|Q_{v, H_J}(s)| > e^{-\alpha_J v /H_J}$. Applying Lemma~\ref{le:Rupest} to $Q_{v, H_J}(s)$ for every $v \in \mathcal{I}_J$ we get
\[
|\mathcal{T}| \ll |\mathcal{I}_J| \cdot T^{2 \alpha_J+o(1)} \cdot T^\eta \cdot X^{o(1)} \ll T^{1/2-\eta} \cdot X^{o(1)}.
\]

Let
\[
\mathcal{T}_{L} = \{t \in \mathcal{T} :  |Q_{v, H}(1+it)| \geq (\log X)^{-100}\}
\]
and
\[
\mathcal{T}_{S} = \{t \in \mathcal{T} :  |Q_{v,H}(1+it)| < (\log X)^{-100}\}.
\]

By Lemma~\ref{le:Hallargevalint},
\begin{align*}
&\sum_{t \in \mathcal{T}_S} |Q_{v,H}(1 + it) R_{v,H}(1 + it)|^2 dt \ll (\log X)^{-200} \cdot \sum_{t \in \mathcal{T}} |R_{v,H}(1 + it)|^2 \\
& \ll (\log X)^{-200} \cdot \Big ( X e^{-v/H} + |\mathcal{T}| T^{1/2} \Big ) \log(2T) \frac{1}{X e^{-v/H}} \ll (\log X)^{-199},
\end{align*}
and thus we can concentrate on the integral over $\mathcal{T}_L$.

By Lemma~\ref{le:Rupest}, we have 
\[
\begin{split}
|\mathcal{T}_{L}| &\ll \exp\left(2\frac{\log (\log X)^{100}}{v/H} \log T + 2 \log (\log X)^{100} + 2\frac{\log T}{v/H} \log \log T\right) \\
&\ll \exp\left(\frac{(\log X)^{1+o(1)}}{\log P}\right) \ll \exp((\log X)^{1/48+o(1)}), 
\end{split}
\]
and by Lemmas~\ref{le:Halappl} and~\ref{le:Sinclexcl} (since $2^J \ll (\log X)^{o(1)}$),
$$
\max_{(\log X)^{1/15} \leq |u| \leq 2T^{1+\varepsilon}} |R_{v, H} (1 + iu)| \ll (\log X)^{-1/16+o(1)} \cdot \frac{\log Q}{\log P}
$$
Thus by Lemma~\ref{le:Hallargevalprimes}, and the Hal\'asz bound above,
\[
\begin{split}
&\sum_{t \in \mathcal{T}_{L}}  |R_{v,H}(1  + it )|^2 \cdot |Q_{v,H}(1 + it )|^2 \\
&\ll (\log X)^{-1/8+o(1)} \left(\frac{\log Q}{\log P}\right)^2 \Big (e^{v/H} + |\mathcal{T}_L| \cdot e^{v/H} \cdot \exp(-(\log X)^{1/5})
\Big ) \cdot \sum_{\substack{e^{v/H} \leq r \leq e^{(v+1)/H} \\ r \in \mathbb{P}}} \frac{1}{r^2 \log r} \\ 
&\ll (\log X)^{-1/8+o(1)} \left(\frac{\log Q}{\log P}\right)^2 \frac{H}{v} \sum_{\substack{e^{v/H} \leq r \leq e^{(v+1)/H} \\ r \in \mathbb{P}}} \frac{1}{r} \ll (\log X)^{-1/8+o(1)} \frac{(\log Q)^2}{(\log P)^4} \frac{1}{H},
\end{split}
\]
where the additional gain comes from the sum over $r \in \mathbb{P}$ saving us an additional $1/v \ll 1/(H \log P)$ (since we are looking at primes in a short interval). 
Combining the above estimates, we get the bound
\[
\begin{split}
\int_{t \in \mathcal{U}} |F(1 + it)|^2 dt &\ll H (\log X)^2 (\log X)^{-1/8+o(1)} \frac{(\log Q)^2}{(\log P)^4} + \left(\frac{T}{X}+1\right) \left(\frac{1}{H}+\frac{\log P}{\log Q}\right) \\
&\ll \left(\frac{T}{X}+1\right) (\log X)^{-1/48+o(1)}.
\end{split}
\]
\subsection{Conclusion} Collecting all the bounds and refering to (\ref{eq:boundplan}) we get
\[
\begin{split}
&\int_{T_0}^{T} |F(1 + it)|^2 dt \\
&\ll \left( \frac{T}{X/Q_1} + 1 \right ) \frac{(\log Q_1)^{1/3}}{P_1^{1/6-\eta}} 
+  \left( \frac{T}{X} + 1 \right ) \left(\sum_{2 \leq j \leq J-1} \frac{1}{j^2 P_1} + \frac{1}{(\log X)^{1/48+o(1)}} \right)\\
&\ll  \left( \frac{T}{X/Q_1} + 1 \right ) \left(\frac{(\log Q_1)^{1/3}}{P_1^{1/6-\eta}} + \frac{1}{(\log X)^{1/50}}\right)
\end{split}
\]
which is the desired bound.
\end{proof}

\section{Proofs of Theorems~\ref{thm:main} and~\ref{th:ThminS}}
\label{sec:MTreduction}
\begin{proof}[Proof of Theorem~\ref{th:ThminS}]
Combining Lemma~\ref{lem:shortintsumtoDirpols} with Proposition~\ref{prop:MainProp} it follows that
\[
\frac{1}{X}\int_X^{2X} \left|\frac{1}{h} \sum_{\substack{x \leq n \leq x + h \\ n \in \mathcal{S}}} f(n) - \frac{1}{h_2} \sum_{\substack{x \leq n \leq x+h_2 \\
n \in \mathcal{S}}} f(n)\right|^2 dx \ll \frac{(\log h)^{1/3}}{P_1^{1/6-\eta}} + \frac{1}{(\log X)^{1/50}},
\]
when $Q_1 \leq h \leq h_2 = \frac{X}{(\log X)^{1/5}}$.
Using Lemma~\ref{le:Lipschitz} together with Lemma~\ref{le:Sinclexcl} we have, for any $X \leq x \leq 2X$, 
\begin{equation}
\label{eq:Lipappl}
\frac{1}{h_2} \sum_{\substack{x \leq n \leq x + h_2 \\ n \in \mathcal{S}}} f(n) = \frac{1}{X} \sum_{\substack{X \leq n \leq 2X \\ n \in \mathcal{S}}} f(n) + O((\log X)^{-1/20+o(1)}), 
\end{equation}
and the claim follows in case $h \leq h_2$. In case $h > h_2$, the claim follows immediately from~\eqref{eq:Lipappl}.
\end{proof}

\begin{proof}[Proof of Theorem~\ref{thm:main}]
Let us start by separating the contribution of $n \not \in \mathcal{S}$, where $\mathcal{S}$ is a set satisfying the conditions in Theorem~\ref{th:ThminS}. We get
\[
\begin{split}
&\Bigg | \frac{1}{h} \sum_{x \leq n \leq x + h} f(n) - \frac{1}{X} \sum_{X \leq n \leq 2X} f(n) \Bigg | \\
&\leq \left| \frac{1}{h} \sum_{\substack{x \leq n \leq x + h \\ n \in \mathcal{S}}} f(n) - \frac{1}{X} \sum_{\substack{X \leq n \leq 2X \\ n \in \mathcal{S}}} f(n) \right| + \frac{1}{h} \sum_{\substack{x \leq n \leq x + h \\ n \not \in \mathcal{S}}} 1 + \frac{1}{X} \sum_{\substack{X \leq n \leq 2X \\ n \not \in \mathcal{S}}} 1.
\end{split}
\]
Let us write
\[
\begin{split}
&\frac{1}{h} \sum_{\substack{x \leq n \leq x + h \\ n \not \in \mathcal{S}}} 1 = 1+O(1/h)-\frac{1}{h} \sum_{\substack{x \leq n \leq x + h \\ n \in \mathcal{S}}} 1 \\
&= \frac{1}{X} \sum_{\substack{X \leq n \leq 2X \\ n \not \in \mathcal{S}}} 1 + \frac{1}{X} \sum_{\substack{X \leq n \leq 2X \\ n \in \mathcal{S}}} 1 +O(1/h)-\frac{1}{h} \sum_{\substack{x \leq n \leq x + h \\ n \in \mathcal{S}}} 1,
\end{split}
\]
so that
\[
\begin{split}
&\Bigg | \frac{1}{h} \sum_{x \leq n \leq x + h} f(n) - \frac{1}{X} \sum_{X \leq n \leq 2X} f(n) \Bigg | \\
&\leq \left| \frac{1}{h} \sum_{\substack{x \leq n \leq x + h \\ n \in \mathcal{S}}} f(n) - \frac{1}{X} \sum_{\substack{X \leq n \leq 2X \\ n \in \mathcal{S}}} f(n) \right| + \left| \frac{1}{h} \sum_{\substack{x \leq n \leq x + h \\ n \in \mathcal{S}}} 1 - \frac{1}{X} \sum_{\substack{X \leq n \leq 2X \\ n \in \mathcal{S}}} 1 \right|
 + \frac{2}{X} \sum_{\substack{X \leq n \leq 2X \\ n \not \in \mathcal{S}}} 1 + O(1/h).
\end{split}
\]

Theorem \ref{th:ThminS} applied to $f(n)$ and to $1$ implies that the first and second terms are both at most $\delta/100$ with 
at most 
\begin{equation}
\label{eq:excthminS}
\ll \frac{X (\log h)^{1/3}}{P_1^{1/6 - \eta}\delta^{2}}+\frac{X}{(\log X)^{1/50}\delta^{2}}
\end{equation}
exceptions. 

By the fundamental lemma of the sieve, for all large enough $X$, 
\begin{align*}
\sum_{\substack{X \leq n \leq 2X \\ n \not \in \mathcal{S}}} 1 \leq
\left(1+\frac{1}{100}\right) X \sum_{j \leq J} \prod_{P_j \leq p \leq Q_j} \Big ( 1 - \frac{1}{p} \Big ) \leq \left(1+\frac{1}{100}\right) X \sum_{j \leq J} \frac{\log P_j}{\log Q_j} 
\end{align*}
Hence we get that
\begin{equation}
\label{eq:difnoSbound}
\Bigg | \frac{1}{h} \sum_{x \leq n \leq x + h} f(n) - \frac{1}{X} \sum_{X \leq n \leq 2X} f(n) \Bigg | \leq \delta/50 + \left(2 + \frac{1}{50}\right) \sum_{j} \frac{\log P_j}{\log Q_j}
\end{equation}
with at most \eqref{eq:excthminS} exceptions. 

To deduce Theorem~\ref{thm:main} we pick an appropriate sequence of intervals $[P_j, Q_j]$. In case $h \leq \exp((\log X)^{1/2})$, we choose $\eta = 1/150$, $Q_1 = h, P_1 = \max\{h^{\delta/4}, (\log h)^{40/\eta}\}$ and $P_j$ and $Q_j$ as in~\eqref{eq:PjQjchoice}.  With this choice the expression in \eqref{eq:difnoSbound} is at most $\delta + 20000\frac{\log \log h}{\log h}$ and the number of exceptions is as claimed.

In case $h > \exp((\log X)^{1/2})$, we choose $\eta = 1/150, Q_1 = \exp((\log X)^{1/2}), P_1 = Q_1^{\delta/4}$ and $P_j$ and $Q_j$ as in~\eqref{eq:PjQjchoice}. This is a valid choice since we can assume $\delta \geq (\log X)^{-1/100}$, so that $P_1 \geq (\log Q_1)^{40/\eta}$.  With this choice the expression in \eqref{eq:difnoSbound} is at most $\delta$ and the number of exceptions is as claimed.

\end{proof}

\section{Proof of Theorems~\ref{th:longints} and~\ref{thm:bilinear}}
\label{sec:longintproof}
Let $\eta_{\xi, v} (x)$ be a smoothing of the indicator function of $[1-v, 1 + v]$
which decays on the segments $[1 - \xi - v, 1 - v]$ and $[1 + v, 1 + \xi + v]$. Precisely,
let
$$
\eta_{\xi, v} (x) = \begin{cases}
1 & \text{ if } 1 - v \leq x \leq 1 + v \\
 (1 + v + \xi - x) / \xi & \text{ if } 1 + v \leq x \leq 1 + \xi + v \\
 (x + v + \xi - 1) / \xi & \text{ if } 1 - \xi - v \leq x \leq 1 - v \\
0 & \text{ otherwise}. 
\end{cases}
$$
We find that
\begin{align*}
\widehat{\eta}_{\xi, v}(s) & := - \int_{0}^{\infty} t^s d \eta_{\xi, v}(t) = - \int_{1-v-\xi}^{1-v} \frac{t^s}{\xi}dt + \int_{1+v}^{1+v+\xi} \frac{t^s}{\xi} dt \\ & = \frac{(1 + \xi + v)^{s + 1} - (1 + v)^{s + 1}}{\xi (s + 1)}
- \frac{(1 - v)^{s + 1} - ( 1 - \xi - v)^{s + 1}}{\xi (s + 1)}.
\end{align*}
Therefore by Mellin inversion,
\begin{equation}
\label{eq:Meleta}
\eta_{\xi, v}(x) = \frac{1}{2\pi i} \int_{1 - i\infty}^{1 + i\infty} \frac{x^{-s}}{s} \cdot \widehat{\eta}_{\xi, v}(s) ds. 
\end{equation}
We are now ready to prove Theorem \ref{th:longints}.  
\begin{proof}[Proof of Theorem \ref{th:longints}]
Let $h_1 = h\sqrt{x}$ and $h_2 = x (\log x)^{-1/5}$. Let $v_j = h_j / x$ and $\xi_j = \delta h_j / x$ for some small $\delta$ to be chosen later. 
Let also
$\eta_j(x) := \eta_{\xi_j, v_j}(x)$  for $j = 1, 2$. 
Consider, 
\begin{align*}
S_j = \sum_{\substack{\sqrt{x} \leq n_1 \leq 2\sqrt{x} \\ n_1, n_2 \in \mathcal{S}}} f(n_1) f(n_2) \eta_{j} \Big ( \frac{n_1 n_2}{x} \Big ). 
\end{align*}
Using \eqref{eq:Meleta}, we see that $S_j$ equals
\begin{align*}
\frac{1}{2\pi i} \int_{1 - i \infty}^{1 + i \infty} M_1(s) M_2(s) x^s \cdot \frac{(1 + \xi_j + v_j)^{s + 1} - (1 + v_j)^{s + 1} - (1 - v_j)^{s + 1} + (1 - \xi_j - v_j)^{s + 1}}{\xi_j \cdot s (s + 1)} ds
\end{align*}
where
$$
M_1(s) := \sum_{\substack{\sqrt{x} \leq n \leq 2\sqrt{x} \\ n \in \mathcal{S}}} \frac{f(n)}{n^s} \quad \text{and} \quad M_2(s) := \sum_{\substack{\sqrt{x}/2 \leq n \leq 2\sqrt{x} \\ n \in \mathcal{S}}} \frac{f(n)}{n^s}
$$
As in the proof of Lemma~\ref{lem:shortintsumtoDirpols} we split the integral in $S_j$ into two parts $U_j$ and $V_j$ according to whether $|t| \leq T_0 := (\log x)^{1/12}$ or not. In $U_j$, we expand each term in the following way, 
$(1 + w)^{1 + s} = 1 + w (1 + s) + \frac{w^2}{2} s ( 1 + s) + O(|w|^3 |s| |s+1| |s - 1|)$ (for $|w| \leq 1/2$ and $\Re s = 1$). This gives, 
\begin{align*}
& x^s \cdot \frac{(1 + \xi_j + v_j)^{s + 1} - (1 + v_j)^{s + 1} - (1 - v_j)^{s + 1} + (1 - \xi_j - v_j)^{s+1}}{\xi_j s (s + 1)} 
\\ & = (\xi_j + 2 v_j) x^s  + O(x (1 + |s|) (\xi_j^3 + v_j^3) / \xi_j) = (2 + \delta) \cdot \frac{h_j}{x} \cdot x^s 
 + O(x \cdot T_0 (h_j/ x)^2 / \delta).  
\end{align*}
so that
$$
\Big | \frac{1}{h_1} U_1 -  \frac{1}{h_2} U_2 \Big | \ll \frac{T_0^2}{\delta} \cdot  \frac{h_2}{x} \ll
\frac{(\log x)^{1/6 - 1/5}}{\delta} \ll \frac{(\log x)^{-1/30}}{\delta}.
$$ 
On the other hand, to bound $V_j$, we notice that (on $\Re s = 1$), 
$$
\frac{|\widehat{\eta}_j(s)|}{|s|} = \Big | \int_{0}^{\infty} t^{s-1} \eta(t) dt \Big | \ll \frac{h_j}{x}  \text{ and } \frac{|\widehat{\eta}_j(s)|}{|s|} \ll \frac{1}{|s| \xi |s+1|} \ll
\frac{x}{\delta h_j} \cdot \frac{1}{1 + |s|^2}.
$$
Therefore splitting the integral $V_j$ at height $x /  h_j$, we get
\begin{align*}
\Big | \frac{1}{h_1} V_1 - \frac{1}{h_2} V_2 \Big | & \leq \frac{1}{\delta} \sum_{j=1}^2 \Big ( 
\int_{1 + iT_0}^{1 + i x / h_j} |M_1(s)M_2(s)| |ds| + \frac{x}{h_j} \max_{T > x/h_j} \frac{1}{T} \int_{1 + iT}^{1 + 2 iT} |M_1(s)M_2(s)| |ds|
\Big ).
\end{align*}
similarly to~\eqref{eq:intsplit}. Using Cauchy-Schwarz inequality and Proposition \ref{prop:MainProp} we thus get the following bound (recall
that $h_1 = h \sqrt{x}, h_2 = x / (\log x)^{1/5}$ and $h \geq Q_1$ by assumptions):
$$
\Big | \frac{1}{h_1}V_1 - \frac{1}{h_2} V_2 \Big | \ll \frac{(\log Q_1)^{1/3}}{\delta P_1^{1/6 - \eta}} + \frac{1}{\delta (\log X)^{1/50}}.  
$$
We now choose $\delta = \max((\log Q_1)^{1/6}/ P_1^{1/12-\eta/2}, (\log X)^{-1/100})$ and notice that
$$
\frac{1}{h_j} \sum_{\substack{\sqrt{x} \leq n_1 \leq 2\sqrt{x} \\ x + h_j \leq n_1 n_2 \leq x + \delta h_j}} 1 \ll \delta.
$$
Therefore
\begin{equation}
\label{eq:bilincomp}
\begin{split}
\frac{1}{h_1} \sum_{\substack{\sqrt{x} \leq n_1 \leq 2\sqrt{x} \\ x - h_1 \leq n_1 n_2 \leq x + h_1 \\ n_1, n_2 \in \mathcal{S}}}
f(n_1) f(n_2) = \frac{1}{h_2} & \sum_{\substack{\sqrt{x} \leq n_1 \leq 2\sqrt{x} \\ x - h_2 \leq n_1 n_2 \leq x + h_2 \\ n_1, n_2 \in \mathcal{S}}} f(n_1)f(n_2)
+ \\
& + O \Big ( \frac{(\log Q_1)^{1/6}}{P_1^{1/12 - \eta/2}} + \frac{1}{(\log X)^{1/100}} \Big ). 
\end{split}
\end{equation}
Finally,
\begin{align*}
\sum_{\substack{x - h_2 \leq n_1 n_2 \leq x +  h_2 \\ \sqrt{x} \leq n_1 \leq 2\sqrt{x} \\ n_1, n_2 \in \mathcal{S}}}
f(n_1) f(n_2) & = \sum_{\substack{\sqrt{x} \leq n_1 \leq 2\sqrt{x} \\ n_1 \in \mathcal{S}}} f(n_1) 
\sum_{\substack{(x - h_2) / n_1 \leq n_2 \leq (x + h_2) / n_1 \\ n_2 \in \mathcal{S}}} f(n_2). 
\end{align*}
and $[(x - h_2) / n_1, (x + h_2) / n_1]$ is an interval of length $\asymp \sqrt{x} / (\log x)^{1/5}$ around $\asymp \sqrt{x}$.
Using Lemma \ref{le:Lipschitz} and Lemma \ref{le:Sinclexcl}, we get 
$$
\frac{1}{h_2/n_1} \sum_{\substack{(x - h_2) / n_1 \leq n_2 \leq (x + h_2) / n_1 \\ n_2 \in \mathcal{S}}} f(n_2) = \frac{2}{\sqrt{x}}
\sum_{\substack{\sqrt{x} \leq n \leq 2\sqrt{x} \\ n \in \mathcal{S}}} f(n) + O ( (\log x)^{-1/20 + o(1)} ),
$$ 
so that
\[
\begin{split}
\frac{1}{h_2} \sum_{\substack{\sqrt{x} \leq n_1 \leq 2\sqrt{x} \\ x - h_2 \leq n_1 n_2 \leq x + h_2 \\ n_1, n_2 \in \mathcal{S}}} f(n_1)f(n_2) &= \frac{2}{\sqrt{x}}
\sum_{\substack{\sqrt{x} \leq n \leq 2\sqrt{x} \\ n \in \mathcal{S}}} f(n) \sum_{\substack{\sqrt{x} \leq n_1 \leq 2\sqrt{x} \\ n_1 \in \mathcal{S}}} \frac{f(n_1)}{n_1} + O ( (\log x)^{-1/20 + o(1)} ) \\
&= 2\log 2 \cdot \Big(\frac{1}{\sqrt{x}}\sum_{\substack{\sqrt{x} \leq n \leq 2\sqrt{x} \\ n \in \mathcal{S}}} f(n)\Big)^2 + O ( (\log x)^{-1/20 + o(1)} )
\end{split}
\]
by partial summation and Lemmas \ref{le:Lipschitz} and \ref{le:Sinclexcl}. The claim follows by combining this with~\eqref{eq:bilincomp}.
\end{proof}

\begin{proof}[Proof of Theorem \ref{thm:bilinear}]
We can assume that $h \leq \exp((\log x)^{1/2})$ since the claim for longer intervals follows by splitting the sum on the left hand side into sums over intervals of length $\sqrt{x}\exp((\log x)^{1/2})$. 

We take $\eta = 1/12$, $Q_1 = h$, and $P_1 = (\log h)^{40/\eta} = (\log h)^{480}$ and for $j \geq 2$, the intervals $[P_j, Q_j]$ as in~\eqref{eq:PjQjchoice}. Arguing as in the proof of Theorem~\ref{thm:main}, and noticing that
\[
\left(\sum_{\sqrt{x} \leq n \leq 2\sqrt{x}}   1\right)^2   = \sum_{\sqrt{x} \leq n_1, n_2 \leq 2 \sqrt{x}}   1   = \left(\sum_{\sqrt{x} \leq n \leq 2\sqrt{x}, n \in \mathcal{S}} 1\right)^2  + \sum_{\substack{\sqrt{x} \leq n_1, n_2 \leq 2 \sqrt{x} \\ n_1 \not \in \mathcal{S} \text{ or } n_2 \not \in \mathcal{S}}} 1,
\]
we obtain
\[
\begin{split}
&\left|\frac{1}{h \sqrt{x} \log 2} \sum_{\substack{x \leq n_1 n_2  \leq x + h \sqrt{x} \\ \sqrt{x} \leq n_1 \leq 2\sqrt{x}}} f(n_1) f(n_2) - \Big ( \frac{1}{\sqrt{x}} \sum_{\sqrt{x} \leq n \leq 2\sqrt{x}} f(n) \Big )^2\right| \\
&\leq \left|\frac{1}{h \sqrt{x}\log 2} \sum_{\substack{x \leq n_1 n_2  \leq x + h \sqrt{x} \\ \sqrt{x} \leq n_1 \leq 2\sqrt{x} \\ n_1, n_2 \in \mathcal{S}}} f(n_1) f(n_2) - \Big ( \frac{1}{\sqrt{x}} \sum_{\substack{\sqrt{x} \leq n \leq 2\sqrt{x} \\ n \in \mathcal{S}}} f(n) \Big )^2\right| \\
&\quad + \left|\frac{1}{h \sqrt{x}\log 2} \sum_{\substack{x \leq n_1 n_2  \leq x + h \sqrt{x} \\ \sqrt{x} \leq n_1 \leq 2\sqrt{x} \\ n_1, n_2 \in \mathcal{S}}} 1 - \Big ( \frac{1}{\sqrt{x}} \sum_{\substack{\sqrt{x} \leq n \leq 2\sqrt{x} \\ n \in \mathcal{S}}} 1 \Big )^2\right| + \frac{2}{x} \sum_{\substack{\sqrt{x} \leq n_1, n_2 \leq 2\sqrt{x} \\ n_1 \not \in \mathcal{S} \text{ or } n_2 \not \in \mathcal{S}}} 1 +  O(1/h).
\end{split}
\]

Now we apply Theorem~\ref{th:longints} to the first two terms and use the fundamental lemma of the sieve to get that
$$
\frac{1}{\sqrt{x}}
\sum_{\substack{\sqrt{x} \leq n \leq 2\sqrt{x} \\ n \notin \mathcal{S}}} 1 \ll \sum_{j} \frac{\log P_j}{\log Q_j} \ll \frac{\log P_1}{\log Q_1} \ll \frac{\log \log h}{\log h}. 
$$
It follows that
\begin{align*}
\frac{1}{h\sqrt{x}\log 2} & \sum_{\substack{x \leq n_1 n_2 \leq x + h\sqrt{x} \\ \sqrt{x} \leq n_1 \leq 2\sqrt{x}}}
f(n_1) f(n_2) = \Big ( \frac{1}{\sqrt{x}} \sum_{\sqrt{x} \leq n \leq 2\sqrt{x}} f(n) \Big )^2 + 
 \\  & + O \Big (\frac{(\log h)^{1/6 + \varepsilon}}{P_1^{1/12 - \eta/2}} + \frac{\log \log h}{\log h} +
(\log x)^{-1/100} \Big ),
\end{align*}
and the claim follows recalling our choices of $\eta$ and $P_1$.
\end{proof}
\section{Proofs of the corollaries}
\subsection{Smooth numbers}
\begin{proof}[Proof of Corollary \ref{cor:smoothinshorts}]
Follows immediately from Theorem~\ref{thm:main} by taking $f$ to be the multiplicative function such that $f(p^\nu) = 1$ for $p \leq x^{1/u}$ and $f(p^\nu) = 0$ otherwise
\end{proof}

\begin{proof}[Proof of Corollary~\ref{cor:smooths}]
The qualitative statement in Corollary~\ref{cor:smooths} would follow immediately from Theorem~\ref{thm:bilinear} together with the Cauchy-Schwarz inequality through the same choice of $f$ as in the previous proof. However, to get a better value for $C(\varepsilon)$, we prove the result using Theeorem~\ref{th:longints} with an appropriate choice of $\mathcal{S}$.

Let $\delta$ be a small positive constant, $\eta \in (0, 1/6)$, and $h$ be fixed but large in terms of $\delta$ and $\eta$. Choose $P_1 = h^{1-\delta}, Q_1 = h$, and for $j \geq 2$ choose
\begin{equation}
\label{eq:PjQjchoiceappl}
P_j = \exp((j/\delta)^{4j} (\log h)^j) \quad \text{and} \quad
Q_j = \exp((j/\delta)^{4j + 2} (\log h)^j).
\end{equation}
This choice satisfies conditions~\eqref{eq:PjQjnottoofar} and~\eqref{eq:PjQjnottooclose}, provided that $h$ is fixed but large enough in terms of $\delta$ and $\eta$.

Notice that with the same choice of $f$ as above, Theorem~\ref{th:longints} implies
that
\begin{align*}
\frac{1}{h \sqrt{x}} \sum_{\substack{x \leq n_1 n_2 \leq x + h \sqrt{x} \\ \sqrt{x} \leq n_1 \leq 2\sqrt{x} \\
n_1, n_2 \in \mathcal{S} \\ n_1, n_2 \text{ }x^{\varepsilon} \text{-smooth}}} 1
\gg \Big ( \frac{1}{\sqrt{x}} \sum_{\substack{\sqrt{x} \leq n \leq 2\sqrt{x} \\ n \in \mathcal{S} \\ n \text{ } x^{\varepsilon}
\text{-smooth}}} 1 \Big )^2 + O \Big ( \frac{(\log Q_1)^{1/6}}{P_1^{1/12 - \eta}} + (\log x)^{-1/100} \Big ).
\end{align*}
The fundamental lemma of the sieve shows that for any $j \leq J$, we have
\[
\sum_{\substack{\sqrt{x} \leq n \leq 2\sqrt{x} \\ p \mid n \implies p \not \in [P_j, Q_j] \\ n \text{ $x^\varepsilon$-smooth}}} 1 \leq
(1+\delta^2)\rho(1/(2\varepsilon)) \sqrt{x} \cdot \frac{\log P_j}{\log Q_j}.
\]
provided that $x$ is large enough, so that 
\[
\begin{split}
\frac{1}{\sqrt{x}}\sum_{\substack{\sqrt{x} \leq n \leq 2\sqrt{x} \\ n \in \mathcal{S} \\ n \text{ } x^{\varepsilon}
\text{-smooth}}} 1 &\geq \frac{1}{\sqrt{x}} \sum_{\substack{\sqrt{x} \leq n \leq 2\sqrt{x} \\ n \text{ } x^{\varepsilon} 
\text{-smooth}}} 1 - \frac{1}{\sqrt{x}} \sum_{j = 1}^J \sum_{\substack{\sqrt{x} \leq n \leq 2\sqrt{x} \\ p \mid n \implies p \not \in [P_j, Q_j] \\ n \text{ $x^\varepsilon$-smooth}}} 1\\
&\geq \rho(1/(2\varepsilon))(1+o(1)) - \sum_{j=1}^J (1+\delta^2)\rho(1/(2\varepsilon)) \cdot \frac{\log P_j}{\log Q_j} \\
&\geq \rho(1/(2\varepsilon))\left(1+o(1)-(1+\delta^2)(1-\delta)-\sum_{j=2}^J \frac{\delta^2}{j^2}\right) \\
&\geq \delta/2 \cdot \rho(1/(2\varepsilon)).
\end{split}
\]
Hence
\begin{align*}
\frac{1}{h \sqrt{x}} \sum_{\substack{x \leq n_1 n_2 \leq x + h \sqrt{x} \\ \sqrt{x} \leq n_1 \leq 2\sqrt{x} \\
n_1, n_2 \in \mathcal{S} \\ n_1, n_2 \text{ }x^{\varepsilon} \text{-smooth}}} 1 \gg \delta^2
\rho(1/(2\varepsilon))^2 + O \Big ( h^{-(1 - \delta)/12 + 1/1000} + (\log x)^{-1/100} \Big ) . 
\end{align*}
Therefore for any small enough $\delta > 0$, and all $x$ large enough, the left-hand side is
$$
\gg \delta^2 \rho(1/\varepsilon)^{1.01} + O \Big  ( h^{-1/12 + 2 \delta + 1/1000} + (\log x)^{-1/100} \Big )
$$
It follows that the lower bound is positive if $h = \rho(1/\varepsilon)^{-13}$ and $\delta, \varepsilon$
are taken small enough. We conclude by using the Cauchy-Schwarz inequality, noting that
\[
\begin{split}
\sqrt{x} &\ll \left(\sum_{\substack{x \leq n \leq x+C\sqrt{x} \\ n \text{ $x^\varepsilon$-smooth}}} 1\right)^{1/2} \left(\sum_{x \leq n \leq x+C\sqrt{x}} \left(\sum_{n_1  n_2 = n} 1\right)^2\right)^{1/2}\\
&\ll\left(\sum_{\substack{x \leq n \leq x+C\sqrt{x} \\ n \text{ $x^\varepsilon$-smooth}}} 1\right)^{1/2} \left(\sqrt{x} (\log x)^4\right)^{1/2} 
\end{split}
\]
and the claim follows.
\end{proof}

\subsection{Signs of multiplicative functions}

\begin{proof}[Proof of Corollary \ref{cor:signchangsinints}]
First notice that the condition that  $f(n) \neq 0$ for a positive proportion of $n$ is equivalent to $\sum_{p, f(p) = 0}\frac{1}{p} < \infty$, and also that we can assume without loss of generality that $f(n) \in \{-1, 0 ,1 \}$. The qualitative statement in Corollary~\ref{cor:signchangsinints} would follow from Theorem~\ref{thm:main} using a slightly simpler variant of the argument below.  However, to get a better bound for the size of the exceptional set, we prove the result using Theeorem~\ref{th:ThminS} with an appropriate choice of $\mathcal{S}$.

Let us choose $P_j$ and $Q_j$ and thus $\mathcal{S}$ as in the proof of Corollary~\ref{cor:smooths} in previous subsection, with $\delta$ small but fixed. By \eqref{eq:Lipsch} together with Lemma~\ref{le:Sinclexcl},
\[
\frac{1}{X} \sum_{\substack{X \leq n \leq 2X \\ n \in \mathcal{S}}} g(n) = \frac{1}{X} \sum_{\substack{n \leq X \\ n \in \mathcal{S}}} g(n) + O((\log X)^{-1/20+o(1)}).
\]
for $g = f$ and $g=|f|$.
Let $p_0^\nu$ be the smallest prime power for which $f(p_0^\nu) = -1$. Now
\[
\sum_{\substack{n \leq X \\ n \in \mathcal{S}}} |f(n)| - f(n) \geq \sum_{\substack{n \leq X/p_0^\nu \\ n \in \mathcal{S}, p_0 \nmid n}} |f(n)|-f(n)+|f(p_0^\nu n)|-f(p_0^\nu n) = 2\sum_{\substack{n \leq X/p_0^\nu \\ n \in \mathcal{S}, p_0 \nmid n}} |f(n)| \gg X
\]
by the fundamental lemma of sieve, similarly to the proof of Corollary~\ref{cor:smooths}.

Applying Theorem~\ref{th:ThminS} to $f(n)$ and $|f(n)|$ we obtain that
\[
\sum_{\substack{x \leq n \leq x+h \\ n \in \mathcal{S}}} |f(n)| - f(n) \gg h
\]
for all but at most 
\begin{equation} \label{exceptionset}
\ll \frac{(\log h)^{1/3}}{h^{(1-\varepsilon)(1/6-\eta)}} + \frac{1}{(\log X)^{1/50}}
\end{equation}
integers $x \in [X, 2X]$.
Hence $f(n)$ is negative in almost all short intervals. Similarly we can show that
$$
\sum_{\substack{x \leq n \leq x + h \\ n \in \mathcal{S}}} |f(n)| + f(n) \gg h
$$
for all but at most (\ref{exceptionset}) exceptional integers $x \in [X, 2X]$. 
Hence $f(n)$ must be positive in almost all short intervals, and the claim follows. 
We actually get that the number of exceptions is $\ll X/h^{1/6-\varepsilon} + (\log X)^{-1/50}$ for any $\varepsilon > 0$.
\end{proof}
It is worth remarking that when $\sum_{f(p) < 0} 1/p < \infty$, one can work out directly the number of sign changes of $f$ up to $x$. For example for non-vanishing completely multiplicative $f$ such that $\sum_{f(p) < 0} 1/p < \infty$, the number of sign changes up to $x$ is asymptotically
$$
x \cdot \Big ( \frac{1}{2} - \frac{1}{2}\prod_{p\colon f(p) < 0} \Big ( 1 - \frac{4}{p+1} \Big ) \Big ) .
$$
Such formulas were pointed out to us by Andrew Granville and Greg Martin, and essentially the formula in general case as well as its proof can be found from a paper by Lucht and Tuttas~\cite{Lucht}.
\begin{proof}[Proof of Corollary \ref{cor:signchanges}]
Follows immediately from the proof of Corollary \ref{cor:signchangsinints}.
\end{proof}
\begin{proof}[Proof of Corollary~\ref{cor:chowla}]
By Corollary~\ref{cor:signchanges}, there is a positive proportion $\delta$ of integers $n$ such that $f(n)f(n+1) \leq 0$. Hence
\[
\sum_{n \leq x} f(n) f(n+1) \leq \sum_{\substack{n \leq x \\ f(n)f(n+1) > 0}} 1 \leq (1-\delta)x.
\]
On the other hand,  
\[
f(n) f(n+1) f(2n) f(2n+1)^2 f(2(n+1)) = (f(2) f(n) f(n+1) f(2n+1))^2  \geq 0,
\]
so that one of $f(n) f(n+1)$, $f(2n)f(2n+1)$ and $f(2n+1)f(2n+2)$ must be non-negative, so that
\[
\sum_{n \leq x} f(n) f(n+1) \geq \sum_{\substack{n \leq x \\ f(n)f(n+1) < 0}} (-1) \geq -(1-\delta)x.
\]
Hence
\begin{equation}
\label{eq:Chowlah=1}
\left|\sum_{n \leq x} f(n) f(n+1)\right| \leq (1-\delta)x.
\end{equation}
For $h \geq 2$,
\[
\begin{split}
\left|\sum_{n \leq x} f(n)f(n+h)\right| &\leq \left|\sum_{\substack{n \leq x \\ h \nmid n}} f(n)f(n+h)\right| + \left|\sum_{\substack{n \leq x \\ h \mid n}} f(n)f(n+h)\right| \\
&\leq\left(1-\frac{1}{h}\right)x + 1+ |f(h)|\left|\sum_{n \leq x/h} f(n)f(n+1)\right| \\
&\leq \left(1-\frac{1}{h}\right)x + 1+ (1-\delta)\frac{x}{h} < (1-\delta(h))x
\end{split}
\]
by \eqref{eq:Chowlah=1}.
\end{proof}

\begin{proof}[Proof of Corollary~\ref{cor:signchangesall}] 
Without loss of generality we can assume that $f(n) \in \{-1, 0, 1\}$. Theorem~\ref{thm:bilinear} implies that for any multiplicative function $g: \mathbb{N} \rightarrow [-1,1]$,
\begin{equation}
\label{eq:glongintwoS}
\frac{1}{h \sqrt{x} \log 2} \sum_{\substack{x \leq n_1 n_2 \leq x + h \sqrt{x} \\ \sqrt{x} \leq n_1 \leq 2\sqrt{x}}} g(n_1) g(n_2) 
= \Big ( \frac{1}{\sqrt{x}} \sum_{\sqrt{x} \leq n \leq 2\sqrt{x}} g(n) \Big )^2 + O ( (\log h)^{-1/100} ). 
\end{equation}
Let us study, for a given $f$,
\[
S^{\pm} = \frac{1}{h \sqrt{x}\log 2}\sum_{\substack{x \leq n_1 n_2 \leq x+h \sqrt{x} \\ \sqrt{x} \leq n_1 \leq 2\sqrt{x}}} (|f(n_1)f(n_2)| \pm f(n_1)f(n_2)).
\]
We will show that $S^+ > 0$ and $S^{-} > 0$. First of these implies that there is $n \in [x, x+h\sqrt{x}]$ such that $f(n) > 0$ (since $f$ is assumed to be completely multiplicative) whereas the second one implies that there is $n \in [x, x+h\sqrt{x}]$ such that $f(n) < 0$.

By~\eqref{eq:glongintwoS}
\[
S^{\pm} = \Big ( \frac{1}{\sqrt{x}} \sum_{\sqrt{x} \leq n \leq 2\sqrt{x}} |f(n)| \Big )^2 \pm \Big ( \frac{1}{\sqrt{x}} \sum_{\sqrt{x} \leq n \leq 2\sqrt{x}} f(n) \Big )^2 + O ( (\log h)^{-1/100} ).
\]
Here the first square is $\gg 1$ by assumption that $f$ is non-vanishing for positive proportion of $n$, so that immediately $S^+ \gg 1$. On the other hand
\[
S^- =  \left(\frac{1}{\sqrt{x}} \sum_{\sqrt{x} \leq n \leq 2\sqrt{x}} (|f(n)| + f(n))\right) \cdot \left(\frac{1}{\sqrt{x}} \sum_{\sqrt{x} \leq n \leq 2\sqrt{x}} (|f(n)| - f(n)) \right) + O ( (\log h)^{-1/100} ).
\]
Arguing as in beginning of proof of Corollary~\ref{cor:signchangsinints},
\[
\frac{1}{\sqrt{x}} \sum_{\sqrt{x} \leq n \leq 2\sqrt{x}} (|f(n)| \pm f(n)) \gg 1,
\]
so that also $S^{-} \gg 1$ and the claim follows.
\end{proof}
It is worth noticing that the case $\sum_{\substack{p \colon f(p) < 0}} \frac{1}{p} < \infty$ is easier than the general case --- actually
it follows from work of Kowalski, Robert and Wu~\cite{KRW07} on $\mathfrak{B}$-free numbers in short intervals that $f$ has a sign change in all intervals $[x, x+x^\theta]$ for any $\theta > 7/17$.

\section*{Acknowledgements}
The authors would like to thank Andrew Granville for many useful discussions on the topic. They would also like to thank the anonymous referee and Joni Ter\"av\"ainen for careful reading of the manuscript.
The first author was supported by the Academy of Finland grants no. 137883 and 138522.

\bibliographystyle{plain}
\bibliography{biblio}

\end{document}